\documentclass[12pt,psamsfonts]{amsart}
\usepackage{fullpage,amssymb,amsfonts,amsmath}
\usepackage{enumerate}
\usepackage[shortlabels]{enumitem}
\usepackage{comment}
\usepackage{hyperref} 
\usepackage[capitalise]{cleveref}
\usepackage{url}
\usepackage{todonotes}

\newtheorem*{corollary*}{Corollary}

\newtheorem{theorem}{Theorem}[section]
\newtheorem{corollary}[theorem]{Corollary}
\newtheorem{proposition}[theorem]{Proposition}
\newtheorem{lemma}[theorem]{Lemma}

\theoremstyle{definition}

\newtheorem*{remark}{Remark}
\newtheorem*{remarks}{Remarks}

\theoremstyle{remark}

\numberwithin{equation}{section}

\newcommand{\cA}{\mathcal{A}}

\newcommand{\cB}{\mathcal{B}}

\newcommand{\kD}{\mathfrak{D}}

\renewcommand{\epsilon}{\varepsilon}

\newcommand{\cE}{\mathcal{E}}

\newcommand{\cJ}{\mathcal{J}}

\newcommand{\cM}{\mathcal{M}}

\newcommand{\Z}{\mathbb{Z}}

\title{Primes represented by positive definite binary quadratic forms}

\author{Asif Zaman}
\address{
Asif Zaman \\
Department of Mathematics \\ 
Stanford University   \\ 
450 Serra Mall, Building 380 \\
Stanford, CA \\
USA \\
94305}
\email{aazaman@stanford.edu}

\date{October 24, 2017}

\begin{document}

\begin{abstract}
Let $f$ be a primitive positive definite integral binary quadratic form of discriminant $-D$ and let $\pi_f(x)$ be the number of primes up to $x$ which are represented by $f$. We prove several types of upper bounds for $\pi_f(x)$ within a constant factor of its asymptotic size: unconditional, conditional on the Generalized Riemann Hypothesis (GRH), and  for almost all discriminants. The key feature of these estimates is that they hold whenever $x$ exceeds a small power of $D$ and, in some cases, this range of $x$ is essentially best possible. In particular, if $f$ is reduced then this optimal range of $x$ is achieved for almost all discriminants or by assuming GRH. We also exhibit an upper bound for the number of primes represented by $f$ in a short interval and  a lower bound for the number of small integers represented by $f$ which have few prime factors.
\end{abstract}

\maketitle
\tableofcontents

%
%
\section{Introduction}

%
%
%
\subsection{Historical review}
The distribution of primes represented by positive definite integral binary quadratic forms is a classical topic within number theory  and has been intensely studied over centuries by many renowned mathematicians, including Fermat, Euler, Lagrange, Legendre, Gauss, Dirichlet, and Weber. A beautiful exposition on the subject and its history can be found in \cite{Cox-2013}. 

We shall refer to positive definite integral binary quadratic forms as simply `forms'. Let
\[
f(u,v) = au^2 + buv + cv^2
\]
be a form of discriminant $-D = b^2 - 4ac$. An integer $n$ is said to be represented by $f$ if there exists $(u,v) \in \Z^2$ such that $n = f(u,v)$. A form is primitive if $a,b,$ and $c$ are relatively prime. The group $SL_2(\Z)$ naturally acts on the set of primitive forms with discriminant $-D$ and two forms are said to be properly equivalent if they belong to the same $SL_2(\Z)$ orbit. A primitive form is reduced if $|b| \leq a \leq c$ and $b \geq 0$ if either $|b| = a$ or $a=c$. Every primitive form  is properly equivalent to a unique reduced form. Most amazingly, the set of primitive forms with discriminant $-D$ modulo proper equivalence can be given a composition law that makes it a finite abelian group $\mathrm{Cl}(-D)$. This is the class group of $-D$ and its size $h(-D)$ is the class number. We refer to each equivalence class of the class group as a form class. 
	
Now, assume $f$ is primitive. The central object of study is the number of primes represented by $f$ up to $x$, denoted
\[
\pi_f(x) = |\{ p \leq x : p = f(u,v) \text{ for some $(u,v) \in \Z^2$ } \}|
\]
for $x \geq 2$. From deep connections established by class field theory, the Chebotarev density theorem \cite{Tschebotareff-1926, LagariasOdlyzko-1977} implies that primes are equidistributed amongst all form classes. Namely, 
\begin{equation}
\pi_f(x) \sim \frac{\delta_f x}{h(-D) \log x}
\label{eqn:PNT_bqfs}
\end{equation}
as $x \rightarrow \infty$, where
\begin{equation}
	\delta_f = 
		\begin{cases}
 			1 & \text{if $f(u,v)$ is properly equivalent to its opposite $f(u,-v)$,}\\
 			\frac{1}{2} & \text{otherwise.}			
		\end{cases}
\end{equation}
Unfortunately, the asymptotic \eqref{eqn:PNT_bqfs} derived from \cite{LagariasOdlyzko-1977} requires $x$ to be exponentially larger than $D$ even without the existence of a Siegel zero. This is unsuitable for many applications. 
Assuming the Generalized Riemann Hypothesis (GRH), one can do much better. It follows from the same work of Lagarias and Odlyzko \cite{LagariasOdlyzko-1977} that, assuming GRH, 
\begin{equation}
	\pi_f(x) = \frac{\delta_f \mathrm{Li}(x)}{h(-D)} + O( x^{1/2} \log(Dx) )
	\label{eqn:PNT_bqfs_GRH}
\end{equation}
for $x \geq 2$. Here $\mathrm{Li}(x) = \int_2^x \frac{1}{\log t}dt \sim \frac{x}{\log x}$. Equation \eqref{eqn:PNT_bqfs_GRH} is nontrivial for $x \geq D^{1+\epsilon}$ and any $\epsilon > 0$. This GRH range for $x$ is essentially optimal for \emph{certain but not necessarily all} forms. Amongst other results, we will partially address this defect in this paper. For further discussion of range optimality and improvements (both conditional and unconditional), see \cref{subsec:optimality,subsec:results}.

Now, to remedy our lack of knowledge about the asymptotics of $\pi_f(x)$ in the unconditional case, one can settle for upper and lower bounds of the shape
\begin{equation}
\frac{x}{h(-D) \log x} \ll \pi_f(x) \ll \frac{x}{h(-D) \log x}
\label{eqn:Hoheisel}
\end{equation}
 in the hopes of improving the valid range of $x$. For uniform lower bounds, we refer the reader to \cite{Fogels-1961,Weiss-1983,KowalskiMichel-2002} and, most recently, \cite{ThornerZaman-2017} wherein it was shown that there exists a prime $p$ represented by $f$ of size at most $O(D^{700})$. 

The focus of this paper, however, is on upper bounds for $\pi_f(x)$. Assuming GRH, there has been no improvement beyond \eqref{eqn:PNT_bqfs_GRH} itself. Unconditionally, a general result of Lagarias, Montgomery and Odlyzko \cite[Theorem 1.4]{LagariasMontgomeryOdlyzko-1979} on the Chebotarev density theorem  yields some progress  but the range of $x$ is still worse than exponential in $D$. Recently, a theorem of Thorner and Zaman \cite{ThornerZaman-2017a} improves upon the aforementioned Chebotarev result and consequently implies that 
\begin{equation}
	\pi_f(x) < \frac{2 \delta_f \mathrm{Li}(x)}{h(-D)} \qquad \text{for $x \geq D^{700}$}
	\label{eqn:BT_bqfs_TZ}
\end{equation}
and $D$ sufficiently large.  Short of excluding a Siegel zero, the constant $2$ is best possible and, moreover, the range of $x$ is polynomial in $D$ similar to \eqref{eqn:PNT_bqfs_GRH}. On the other hand, the quality of exponent $700$ in \eqref{eqn:BT_bqfs_TZ} leaves much to be desired when compared to the GRH exponent of $1+\epsilon$. Both \cite{LagariasMontgomeryOdlyzko-1979,ThornerZaman-2017a} carefully study the zeros of Hecke $L$-functions to prove their respective results  whereas a more recent paper of Debaene \cite{Debaene-2016} uses a lattice point counting argument and Selberg's sieve to subsequently establish another such Chebotarev-type theorem. His result implies a weaker variant of inequality \eqref{eqn:BT_bqfs_TZ} but for a greatly improved range of $x \geq D^{9/2+\epsilon}$. Broadly speaking, we will specialize Debaene's strategy to positive definite binary quadratic forms. 

An alternate formulation that expands the valid range of upper (and lower) bounds for $\pi_f(x)$ can be obtained by averaging over discriminants or forms, in analogy with the famous theorems of Bombieri--Vinogradov and Barban--Davenport--Halberstam on primes in arithmetic progressions. As part of his Ph.D. thesis, Ditchen \cite{Ditchen-2013,Ditchen-2013a} achieved such elegant statements which we present in rough terms for simplicity's sake. Namely, for 100\% of fundamental discriminants $-D \not\equiv 0 \pmod{8}$ and all of their forms $f$, he proved
\begin{equation}
\pi_f(x) = \frac{\delta_f x}{h(-D) \log x} + O_{\epsilon}\Big(\frac{x}{h(-D) (\log x)^2} \Big) 
\label{eqn:Ditchen} 
\end{equation}
provided $x \geq D^{20/3+\epsilon}$. He also showed for 100\% of fundamental discriminants $-D \not\equiv 0 \pmod{8}$ and 100\% of their forms, equation \eqref{eqn:Ditchen} holds for $x \geq D^{3+\epsilon}$. 

Finally, before we discuss an optimal range for upper bounds of $\pi_f(x)$ and the details of our results, we give a somewhat imprecise flavour of what we have shown. One should compare with \eqref{eqn:PNT_bqfs_GRH}, \eqref{eqn:BT_bqfs_TZ}, \eqref{eqn:Ditchen}, and their associated papers \cite{LagariasOdlyzko-1977,ThornerZaman-2017a,Debaene-2016,Ditchen-2013}.

\begin{corollary*}
	Let $f(u,v) = au^2 + buv + cv^2$ be a reduced positive definite integral binary quadratic form of discriminant $-D$. If GRH holds and $\epsilon > 0$ then
	\[
	\pi_f(x) \ll_{\epsilon} \frac{x}{h(-D) \log x} \qquad \text{for $x \geq (D/a)^{1+\epsilon}$}. 
	\]
	This estimate holds unconditionally for 100\% of discriminants $-D$. Unconditionally and uniformly over all discriminants $-D$, the same upper bound for $\pi_f(x)$ holds for $x \geq (D^2/a)^{1+\epsilon}$. 
\end{corollary*}

%
%
%
\subsection{Optimal range}
\label{subsec:optimality} What is the minimal size of $x$ relative to $D$ for which one can reasonably expect \eqref{eqn:Hoheisel} to hold for all forms $f$? We have seen GRH implies $x \geq D^{1+\epsilon}$ is valid and, in one sense, this range is best possible. For example, the form 
\[
f(u,v) = u^2 + \frac{D}{4}v^2
\]
with $D \equiv 0 \pmod{4}$ does not represent any prime $< D/4$. The reason is simple: the coefficients are simply too large. On the other hand, what about forms 
\[
f(u,v) = au^2 + buv + cv^2
\] 
for which all of the coefficients are small compared to $D$, say $a,b,c \ll \sqrt{D}$? The \emph{GRH range} $x \geq D^{1+\epsilon}$ is insensitive to the size of the form's coefficients. Since primitive forms are properly equivalent to a reduced form and all forms in the same form class represent the same primes, it is enough to consider a reduced form $f$ so the coefficients necessarily satisfy $|b| \leq a \leq c$. Consider the sum
\[
\sum_{n \leq x} |\{ (u,v) \in \Z^2 : n = f(u,v) \}|.
\]
If $x < c$ then, as $f$ is reduced, the only terms $n$ contributing to the above sum are of the form $n = f(u,0) = au^2$. The sum is therefore  equal to $2\sqrt{x/a} + O(1)$ and at most one prime value $n$ contributes to its size. On the other hand, if $x \geq c$ then the above sum is of size $O(x/\sqrt{D})$. One might therefore reasonably suspect that primes will appear with natural frequency for $x \geq c^{1+\epsilon}$ and any $\epsilon > 0$. As $c \asymp D/a$, it is conceivable for reduced forms $f$ to satisfy \eqref{eqn:PNT_bqfs_GRH} in (what we will refer to as) the \emph{optimal\footnotemark \,range} $x \geq (D/a)^{1+\epsilon}$ for any $\epsilon > 0$.  \footnotetext{Of course, one could potentially refine the factor of $\epsilon$ but that is not our goal.}

Can we expect to obtain any kind of unconditional or GRH-conditional bounds for $\pi_f(x)$ in the optimal range? For the sake of comparison, we turn to primes in arithmetic progressions since estimates like \eqref{eqn:BT_bqfs_TZ} are relatives of the classical Brun--Titchmarsh inequality. A version due to Montgomery and Vaughan \cite{MontgomeryVaughan-1973} states for $(a,q) = 1$ and $x > q$ that
\begin{equation}
\pi(x;q,a) < \frac{2x}{\varphi(q) \log (x/q)}. 
\label{eqn:BT_APs}
\end{equation}
Here $\pi(x;q,a)$ represents the number of primes up to $x$ congruent to $a \pmod{q}$. The range $x > q$ is clearly best possible which inspires the possibility of success for a similar approach to primes represented binary quadratic forms. Our goal is to give upper bounds for $\pi_f(x)$ as close to the optimal range as possible. 

%
%
%
\subsection{Results}
\label{subsec:results}

Our first result is a uniform upper bound for $\pi_f(x)$, both unconditional and conditional on GRH. For a discriminant $-D$,  let $\chi_{-D}( \, \cdot \,) = \big(\frac{-D}{\, \cdot \,}\big)$ denote the corresponding Kronecker symbol which is a quadratic Dirichlet character.

%
%
%
%
\begin{theorem}
	\label{theorem:BT_uniform_simplified}
	Let $f(u,v) = au^2 + buv + cv^2$ be a reduced positive definite integral binary quadratic form with discriminant $-D$ and let $\epsilon > 0$ be arbitrary. Set
	\begin{equation}
		\phi = \phi(\chi_{-D}) := \begin{cases}
			0 & \text{assuming $L(s,\chi_{-D})$ satisfies GRH,} \\ 
			\frac{1}{4} & \text{unconditionally}. 
			\end{cases}
		\label{def:phi}
	\end{equation}
	If $x \geq (D^{1+4\phi}/a)^{1+\epsilon}$ then
	\begin{equation}
 	\pi_f(x) <  \frac{4}{1-\theta} \cdot \frac{\delta_f x}{h(-D) \log x} \Big\{ 1 + O_{\epsilon}\big(\frac{\log\log x}{\log x} \big)  \Big\}, 
 	\label{eqn:BT_uniform_simplifed}
	\end{equation}
	where
	\[
	\theta = \theta_{\phi} = \big(1+2\phi + \frac{\epsilon}{2} \big) \frac{\log D}{\log x} - \frac{\log a}{\log x}. 
	\]

\end{theorem}
%
\begin{remarks}
	~
	\begin{enumerate}[(i)]
		\item If $\phi = 0$ then $0 < \theta < 1$. Similarly, if $\phi = \frac{1}{4}$ then $0 < \theta < \frac{3}{4}$. The constant $\phi$ is associated with bounds for $L(s,\chi_{-D})$ in the critical strip. In particular, $\phi = \frac{1}{2}$ corresponds to the usual convexity estimate. 
		\item Unconditionally, one should compare this estimate with \eqref{eqn:BT_bqfs_TZ} and the related works \cite{ThornerZaman-2017a,Debaene-2016}.  \cref{theorem:BT_uniform_simplified} achieves the upper bound in \eqref{eqn:Hoheisel} with the range $x \geq (D^2/a)^{1+\epsilon}$ which improves over the prior range $x \geq D^{9/2+\epsilon}$ implied by \cite{Debaene-2016}. In fact, when $a \gg D^{1/2}$, \cref{theorem:BT_uniform_simplified}'s range becomes $x \gg D^{3/2+\epsilon}$ which is fairly close to the classical GRH range $x \geq D^{1+\epsilon}$. Of course,  inequality \eqref{eqn:BT_uniform_simplifed} has a weaker implied constant $\frac{4}{1-\theta}$ instead of $2$ as in \eqref{eqn:BT_bqfs_TZ}.
		\item Assuming GRH, we obtain the desired range $x \geq (D/a)^{1+\epsilon}$ discussed in \cref{subsec:optimality}. Note we only assume GRH for the quadratic Dirichlet $L$-function $L(s,\chi_{-D})$ whereas \eqref{eqn:PNT_bqfs_GRH} assumes GRH for the collection of Hecke $L$-functions associated to the corresponding ring class field. 
	\end{enumerate}
\end{remarks}

Since every primitive form is properly equivalent to a reduced form, we may ignore the dependence on the coefficient $a$ in \cref{theorem:BT_uniform_simplified} to obtain the following simplified result. 

%
%
%
%
\begin{corollary}
	\label{corollary:BT_uniform_simplified}
	Let $f$ be a primitive positive definite integral binary quadratic form with discriminant $-D$ and let $\epsilon > 0$ be arbitrary. If $x \geq D^{2+\epsilon}$ then
	\[
	\pi_f(x) < 8  \cdot \frac{\delta_f x}{h(-D) \log x}\Big\{ 1 + O_{\epsilon}\Big(\frac{\log\log x}{\log x}\Big) \Big\}.  
	\]
\end{corollary}

We will actually prove a more general version of \cref{theorem:BT_uniform_simplified} that allows us to estimate the number of primes represented by $f$ in a short interval.  

%
%
%
%
\begin{theorem}
	\label{theorem:BT_uniform}
	Let $f(u,v) = au^2 + buv + cv^2$ be a reduced positive definite integral binary quadratic form with discriminant $-D$ and let $\epsilon > 0$ be arbitrary. Let $\phi$ be defined by \eqref{def:phi}. 
	If $\big(\dfrac{D^{1+4\phi}}{a}\big)^{1/2+\epsilon} x^{1/2+\epsilon} \leq y \leq x$ then
	\begin{equation}
 	\pi_f(x) - \pi_f(x-y) <  \frac{2}{1-\theta'} \cdot \frac{\delta_f y}{h(-D) \log y} \Big\{ 1 + O_{\epsilon}\big(\frac{\log\log y}{\log y} \big)  \Big\}, 
 		\label{eqn:BT_short_intervals}
	\end{equation}
	where
	\[
	\theta' = \theta'_{\phi} = \frac{\log x}{2\log y} + \big(\frac{1}{2} + \phi + \frac{\epsilon}{4} \big) \frac{\log D}{\log y} -   \frac{\log a}{2\log y}.
	\]
\end{theorem}
%
\begin{remarks}
~
\begin{enumerate}[(i)]
	\item Assuming GRH, \eqref{eqn:PNT_bqfs_GRH} implies that
		\[
		\pi_f(x) - \pi_f(x-y) \ll \frac{\delta_f y}{h(-D) \log y}
		\]
		for $(Dx)^{1/2+\epsilon} \leq y \leq x$. \cref{theorem:BT_uniform} yields an unconditional upper bound of comparable strength and,  depending on the size of the coefficients of $f$, implies a GRH upper bound for slightly shorter intervals and smaller values of $x$. 
 	\item If $f$ is only assumed to be primitive then the same statement holds by setting $a=1$ in the condition on $y$ and the value of $\theta'$. 
	\item \cref{theorem:BT_uniform_simplified} follows by setting $y = x$. 
\end{enumerate}
\end{remarks}

For any given discriminant $-D$, the unconditional versions ($\phi = \frac{1}{4}$) of \cref{theorem:BT_uniform,theorem:BT_uniform_simplified} fall slightly shy of their GRH counterparts ($\phi = 0$). Averaging over all discriminants, we show that the GRH quality estimates hold almost always. For $Q \geq 3$, define
\begin{equation}
\kD(Q) := \{ \text{discriminants $-D$ with $3 \leq D \leq Q$} \}.
\label{def:discriminants}
\end{equation}
Here and throughout, a discriminant $-D$ is that of a positive definite integral binary quadratic form. Thus, $-D$ is a negative integer $\equiv 0$ or $1 \pmod{4}$. 

%
%
%
%
\begin{theorem}
	\label{theorem:BT_average}
	Let $Q \geq 3$ and $0 < \epsilon < \frac{1}{8}$. For all except at most $O_{\epsilon}(Q^{1-\frac{\epsilon}{10}})$ discriminants $-D \in \kD(Q)$, the statements in \cref{theorem:BT_uniform_simplified,theorem:BT_uniform} hold unconditionally with $\phi = 0$. 
\end{theorem}
%

\begin{remark} When considering upper bounds for $\pi_f(x)$, this improves over \eqref{eqn:Ditchen} in several aspects. First, the desired range discussed in \cref{subsec:optimality} is achieved on average. Second, we did not utilize any averaging over forms, only their discriminants. Finally, there are no restrictions on the family of discriminants; they need not be fundamental or satisfy any special congruence condition. 
\end{remark}

The underlying strategy to establish \cref{theorem:BT_uniform,theorem:BT_uniform_simplified,theorem:BT_average} rests on a natural two-step plan. First, estimate the congruence sums
\begin{equation}
|\cA_{\ell}| = \sum_{\substack{n \leq x \\ \ell \mid n}} r_f(n),
\label{eqn:congruence_sums_intro}
\end{equation}
where $\ell$ is squarefree and $r_f(n) = |\{ (u,v) \in \Z^2 : n = f(u,v) \}|$ is the number of representations of the integer $n$ by the form $f$. Second, apply Selberg's upper bound sieve. Our application of the sieve is fairly routine but calculating the congruence sums with sufficient  precision poses some difficulties.

Inspired by a beautiful paper of Granville and Blomer \cite{BlomerGranville-2006} wherein they carefully study the moments of $r_f(n)$, we determine the congruence sums via geometry of numbers methods. However, for their purposes, only a simple well-known estimate \cite[Lemma 3.1]{BlomerGranville-2006} for the first moment $\sum_{n \leq x} r_f(n)$ was necessary. We execute a more refined analysis of the first moment and related quantities $|\cB_{\ell}(m)|$ (see \eqref{def:sieve_sequence_primitive} for a definition) using standard arguments with the sawtooth function. Afterwards, the main technical hurdle is to carefully decompose the congruence sums $|\cA_{\ell}|$ into a relatively small number of disjoint quantities $|\cB_{\ell}(m)|$. This allows us to apply our existing estimates for $|\cB_{\ell}(m)|$ (see \cref{lemma:primitve_congruence_sum_fibre}) while simultaneously controlling the compounding error terms. We achieve this in \cref{prop:local_density}; see \cref{sec:congruence_sums,sec:local_densities} for details on this argument. When finalizing the proofs of \cref{theorem:BT_uniform_simplified,theorem:BT_uniform,theorem:BT_average}, the various valid ranges for $x$ are determined by relying on character sum estimates like Burgess's bound, Heath-Brown's mean value theorem for quadratic characters \cite{Heath-Brown-1995}, and Jutila's zero density estimate \cite{Jutila-1975}.   

Since we have calculated the congruence sums \eqref{eqn:congruence_sums_intro} in \cref{prop:local_density}, we thought it may be of independent interest to examine its performance in conjunction with a lower bound sieve. By a direct application of the beta sieve, we show: 

%
%
%
%
\begin{theorem}
	\label{theorem:almost_primes}
	Let $f(u,v) = au^2 + buv + cv^2$ be a reduced positive definite integral binary quadratic form with discriminant $-D$. For every integer $k \geq 10$, the number of integers represented by $f$ with at most $k$ prime factors is
	\[
	\gg \frac{x}{\sqrt{D} (\log x)^2}  \qquad \text{	for $x \geq \Big(\frac{D}{a}\Big)^{1 + \frac{49}{5k-49}}$. } 
	\]
\end{theorem}
\begin{remarks}
~
\begin{enumerate}[(i)]
	\item 	One can formulate this statement in a slightly weaker alternative fashion: for every $\epsilon > 0$ and $x \geq (D/a)^{1+\epsilon}$, the number of integers represented by $f$ with at most $O_{\epsilon}(1)$ prime factors is $\gg \frac{x}{\sqrt{D}(\log x)^2}$. 	
	\item It is unsurprising that $f$ represents many integers with few prime factors as this follows from standard techniques in sieve theory. The form $f$  even represents primes of size $O(D^{700})$ by \cite{ThornerZaman-2017}. However, the key feature of \cref{theorem:almost_primes} is that the size of these integers with few prime factors is very small (in a best possible sense per \cref{subsec:optimality}). 
\end{enumerate}
\end{remarks}

%

%
%
%

%
%
%
\subsection*{Acknowledgements} I would like to thank John Friedlander and Jesse Thorner for their encouragement and many helpful comments on an earlier version of this paper. I am also grateful to Kannan Soundararajan for several insightful conversations and for initially motivating me to pursue this approach. Part of this work was completed with the support of an NSERC Postdoctoral Fellowship. 

%
%
%
\section{Notation and conventions}
For each of the asymptotic inequalities $F \ll G$, $F = O(G)$, or $G \gg F$, we mean there exists a constant $C > 0$ such that $|F| \leq C G$. We henceforth adhere to the convention that all implied constants in all asymptotic inequalities are absolute with respect to all parameters and are effectively computable.  If an implied constant depends on a parameter, such as $\epsilon$, then we use $\ll_{\epsilon}$ and $O_{\epsilon}$ to denote that the implied constant depends at most on $\epsilon$. 

Throughout the paper, 
\begin{itemize}
	\item A form refers to a positive definite binary integral quadratic form. 
	\item $-D$ is the discriminant of a positive definite integral binary quadratic form, so $-D$ is any negative integer $\equiv 0$ or $1 \pmod{4}$. It is not necessarily fundamental. 
	\item 	$f(u,v) = au^2 + buv + cv^2 \in \mathbb{Z}[u,v]$ is a positive definite integral binary quadratic form of discriminant $-D$. It is not necessarily primitive or reduced. 
	\item  $\chi_{-D}( \, \cdot \,) = (\frac{-D}{\, \cdot \,})$ is the Kronecker symbol attached to $-D$. 
	\item $\mathfrak{D}(Q)$ is the set of discriminants $-D$ with $3 \leq D \leq Q$. 
	\item $\Delta$ is a (negative) fundamental discriminant of a form.
	\item $\tau_k$ is the $k$-divisor function and $\tau = \tau_2$ is the divisor function. 
	\item $\varphi$ is the Euler totient function. 
	\item $[s,t]$ is the least common multiple of integers $s$ and $t$. 
	\item $(s,t)$ is the greatest common divisor of integers $s$ and $t$. However, we may abuse notation and sometimes refer to a lattice point $(u,v) \in \Z^2$ but this will be made clear from context (e.g. with the set membership symbol). 
\end{itemize}

%
%
\section{Elementary estimates}
First, we establish a standard result employing the Euler-Maclaurin summation formula \cite[Lemma 4.1]{IwaniecKowalski-2004}, which will later allow us to counts lattice points inside an ellipse.

%
%
\begin{lemma}
	\label{lemma:sqrt_average}
	For $W \geq 1$, 
	\[
	\sum_{1 \leq w \leq W} \sqrt{W^2 - w^2} = \frac{\pi W^2}{4} - \frac{W}{2} + O(\sqrt{W}). 
	\]
\end{lemma}
%

\begin{proof}
	Set $G(w) = \sqrt{W^2-w^2}$. By partial summation, observe that
\begin{equation}
\begin{aligned}
\sum_{1 \leq w \leq W} G(w) 
& = - \int_0^W t G'(t) dt + \int_0^W \psi(t) G'(t) dt - \tfrac{1}{2}G(0),
\end{aligned}
\label{eqn:sqrt_average_step1}
\end{equation}
where $\psi(t) = t-\lfloor t\rfloor-1/2$ is the sawtooth function. 
For the first integral, notice
\begin{equation}
- \int_0^W tG'(t) dt =  \int_0^W \frac{t^2}{\sqrt{W^2-t^2}} dt = \frac{\pi W^2}{4}.  
\label{eqn:sqrt_average_integral1}
\end{equation}
For the second integral, we use the Fourier expansion (see e.g. \cite[Equation (4.18)]{IwaniecKowalski-2004})
\begin{equation*}
\begin{aligned}
\psi(x) 
& = - \sum_{1 \leq n \leq N} (\pi n)^{-1} \sin(2\pi n x) + O( (1+||x||N)^{-1}),
\end{aligned}
\end{equation*}
where $||x||$ is the distance of $x$ to the nearest integer.  It follows that
\begin{equation}
\begin{aligned}
\int_0^W \psi(t) G'(t)dt 
& = 2 \sum_{1 \leq n \leq N} \int_0^W G(t) \cos(2\pi nt) dt + O\Big( \int_0^W \dfrac{|G'(t)|}{1+||t||N} dt\Big) 
\end{aligned}
	\label{integral_psi_dF}
\end{equation}
after integrating by parts. Using a computer algebra package or table of integrals,
\begin{equation}
	\begin{aligned}
		\int_0^W G(t) \cos(2\pi nt) dt 
			& =  \int_0^W \sqrt{W^2-t^2} \cos(2\pi n t) dt 
			  = W \cdot  \frac{J_1(2\pi nW)}{4n},
	\end{aligned}
	\label{integral_Fexp}
\end{equation}
where $J_{\nu}(z)$ is the Bessel function of the first kind. Recall that, for $z > 1 + |\nu|^2$,
\[
J_{\nu}(z) =  \big(\frac{1}{2}z\big)^{\nu} \sum_{m=0}^{\infty} \frac{(-1)^m (\tfrac{1}{4} z^2)^m}{m!\Gamma(m+\nu+1)} = \Big( \frac{2}{\pi z} \Big)^{1/2} \big( \cos(z-\tfrac{1}{2}\nu \pi - \tfrac{1}{4}\pi) + O\big(\frac{1+|\nu|^2}{z}\big) \big).
\]
When summing \eqref{integral_Fexp} over $1 \leq n \leq N$, it follows that
\begin{equation}
 	\begin{aligned}
	2 \sum_{1 \leq n \leq N} \int_0^W G(t) \cos(2\pi nt) dt 
		 & = \frac{W}{4} \cdot 2\sum_{1 \leq n \leq N} \frac{J_1(2\pi nW)}{n}  	\\
		 & = \frac{\sqrt{W}}{2\pi } \sum_{1 \leq n \leq N} \Big( \frac{\cos(2\pi nW - 3\pi/4)}{n^{3/2}} + O\big( \frac{1}{n^{5/2} W} \big) \Big) \\
		 & \ll \sqrt{W}
 	\end{aligned}
 	\label{fourier_F_exp}
\end{equation}
uniformly over $N$. Taking $N \rightarrow \infty$, we conclude from \eqref{integral_psi_dF} and \eqref{fourier_F_exp} that
\begin{equation}
\int_0^W \psi(t) G'(t)dt \ll \sqrt{W}.
\label{eqn:integral_psi_dF_final}
\end{equation}
Combining \eqref{eqn:sqrt_average_step1}, \eqref{eqn:sqrt_average_integral1}, and the above yields the result. 
\end{proof}
%

Next, in \cref{lemma:weighted_Dirichlet_sum}, we calculate some weighted average values of the Dirichlet convolution
\[
(1 \ast \chi)(n) = \sum_{d \mid n} \chi(d)
\]
for any quadratic Dirichlet character $\chi$. Better estimates and certainly simpler proofs are available via Mellin inversion, but it is useful for us to explicitly express the error terms using the character sum quantity
\[
S_{\chi}(t) := \sum_{n \leq t} \chi(n). 
\]
This feature will give us flexibility and streamline the proofs of \cref{theorem:BT_uniform,theorem:BT_average} each of which use different bounds for character sums. One utilizes uniform bounds (conditional and unconditional) whereas the other applies average bounds. 
%
%
\begin{lemma}
	\label{lemma:weighted_Dirichlet_sum}	
	Let $\chi \pmod{D}$ be a quadratic Dirichlet character. For $x \geq 3$, 
	\begin{equation}
			\label{eqn:weighted_Dirichlet_sum}	
			\sum_{n \leq x} (1 \ast \chi)(n) \Big(1-\frac{n}{x}\Big) = \frac{x}{2} L(1,\chi) + O(\cE_0(x; \chi)),
	\end{equation}
	where 
	\[
	\cE_0(x; \chi) :=  \min_{1 \leq y \leq x} \Big(\frac{y^2}{x} + |S_{\chi}(y)| +  x \int_y^{\infty} \frac{|S_{\chi}(t)|}{t^2} dt\Big)   
	\]
	and $S_{\chi}(t) = \sum_{n \leq t} \chi(n)$. Moreover, 
	\begin{equation}
	\begin{aligned}
				\sum_{n \leq x} \frac{(1 \ast \chi)(n)}{n} \Big(1-\frac{n}{x}\Big)
				& = L(1,\chi) (\log x + \gamma-1)  + L'(1,\chi) + O( \cE_1(x; \chi) ), 
		\label{eqn:log_weighted_Dirichlet_sum}
	\end{aligned}	
	\end{equation}
	where
	\[
	\cE_1(x; \chi) := \min_{1 \leq y \leq x} \Big(\frac{y}{x} + \log x \int_y^{\infty} \frac{|S_{\chi}(t)| \log t }{t^2} dt   \Big).
	\]
\end{lemma}
%

\begin{remark} Bounds for $\cE_0(x; \chi)$ and $\cE_1(x; \chi)$ can be found in \cref{sec:average_bounds,sec:uniform_bounds}	. 
\end{remark}

%
\begin{proof} The proofs of these facts are standard but we include the details for the sake of completeness. Equation \eqref{eqn:weighted_Dirichlet_sum} is a more precise version of \cite[Section 4.3.1, Exercise 3]{MontgomeryVaughan-2007} which originated from work of Mertens. Before we proceed, recall by partial summation that 
	\begin{equation}
	\begin{aligned}
	\sum_{d \leq y} \frac{\chi(d)}{d} 
		& = L(1,\chi) - \int_y^{\infty} \frac{S_{\chi}(t)}{t^2} dt, \\
	\sum_{d \leq y} \frac{\chi(d) \log d}{d}
		& =  -L'(1,\chi) - \int_y^{\infty} \frac{S_{\chi}(t)(\log t-1)}{t^2} dt.
	\label{eqn:approximation_L1chi} 
	\end{aligned}
	\end{equation}
	We will use these estimates in what follows. Let $A(u) = \sum_{n \leq u} (1-\frac{n}{u})$. One can verify that
	\begin{equation}
	A(u) = \frac{u}{2} - \frac{1}{2} \int_0^u \{ t \} dt =  \frac{u}{2} - \frac{1}{2} + O(u^{-1}). 
	\label{eqn:weighted_partial_sums}
	\end{equation}
	From the above integral formula for $A(u)$, it is straightforward to check that $A(u)$ is continuous and if $u > 1$ is not an integer then
	\begin{equation}
	A'(u) = \frac{[u]([u]+1)}{2u^2} \ll 1. 
	\label{eqn:A_derivative}
	\end{equation}
	In particular, $A(u)$ is increasing and absolutely continuous. 
	Now, to calculate the sum in \eqref{eqn:weighted_Dirichlet_sum}, we use Dirichlet's hyperbola method with a parameter $1 \leq y \leq x$. Namely,
	\begin{equation}
	\begin{aligned}
		\Sigma := \sum_{n \leq x} (1 \ast \chi)(n) \Big(1-\frac{n}{x}\Big)
			& = \sum_{d \leq y} \chi(d) A(x/d) + \sum_{y < d \leq x} \chi(d) A(x/d) \\
			& = \Sigma_1 + \Sigma_2,  
	\end{aligned}
	\label{eqn:weighted_partial_sums_Sigma}
	\end{equation}
	say. For $\Sigma_1$, we see by \eqref{eqn:weighted_partial_sums} that
	\begin{equation}
		\Sigma_1 
			= \frac{x}{2} \sum_{d \leq y} \frac{\chi(d)}{d} - \frac{1}{2} \sum_{d \leq y} \chi(d) + O(y^2/x). 
		\label{eqn:weighted_partial_sums_Sigma_1}
	\end{equation}
	For $\Sigma_2$, since $A$ is an absolutely continuous, decreasing, non-negative function and $A(1) = 0$, it follows by partial summation that
	\begin{equation*}
	\begin{aligned}
	\Sigma_2 
		 = \int_y^x A(x/t) dS_{\chi}(t) 
		 = A(1) S_{\chi}(x) + \int_y^x S_{\chi}(t) dA(x/t)	
		& = -\int_y^x S_{\chi}(t) A'(x/t) x t^{-2} dt. 
	\end{aligned}
	\end{equation*}
	Thus, by  \eqref{eqn:A_derivative},  
	\begin{equation*}
	\begin{aligned}
		|\Sigma_2| 
			& \ll x \Big(\int_y^x \frac{|S_{\chi}(t)|}{t^2} dt \Big). 
	\end{aligned}
	\end{equation*}
	Combining the above estimate, \eqref{eqn:weighted_partial_sums_Sigma_1}, and \eqref{eqn:approximation_L1chi} into \eqref{eqn:weighted_partial_sums_Sigma} yields \eqref{eqn:weighted_Dirichlet_sum}.

	We prove \eqref{eqn:log_weighted_Dirichlet_sum} similarly. For $u \geq 1$, define 
	\begin{equation}
	B(u) := \sum_{n \leq u} \frac{1}{n}\big(1-\frac{n}{u}\big) = \log u + \gamma - 1 + O(u^{-1}).
	\label{eqn:log_weighted_partial_sums}
	\end{equation}
	One can verify that $B(u)$ is a non-negative, increasing, and absolutely continuous function of $u$. Also, if $u > 1$ is not an integer then
	\begin{equation}
	B'(u) \ll \frac{1}{u}. 
	\label{eqn:B_derivative}
	\end{equation}
	For some parameter $1 \leq y \leq x$, 
	\begin{equation}
	\begin{aligned}
			\Sigma' := \sum_{n \leq x} \frac{(1 \ast \chi)(n)}{n}\Big(1-\frac{n}{x}\Big)
				& = \sum_{d \leq y} \frac{\chi(d)}{d} B(x/d)  + \sum_{y < d \leq x} \frac{\chi(d)}{d} B(x/d) \\
				& = \Sigma_1' + \Sigma_2',
	\end{aligned}
	\label{eqn:log_weighted_partial_sums_Sigma}
	\end{equation}
	say. To calculate $\Sigma_1'$, we apply \eqref{eqn:log_weighted_partial_sums} and deduce that
	\begin{equation}
	\Sigma_1' = (\log x + \gamma-1) \sum_{d \leq y} \frac{\chi(d)}{d} - \sum_{d \leq y} \frac{\chi(d) \log d}{d} +  O(y/x). 
	\label{eqn:log_weighted_partial_sums_Sigma_1}
	\end{equation}
	For $\Sigma_2'$, set $\tilde{S}_{\chi}(u) := \sum_{y < d \leq u} \frac{\chi(d)}{d}$.  Since $B$ is an absolutely continuous, non-negative, increasing function and $B(1) = 0$, we similarly conclude that
	\begin{equation*}
	\begin{aligned}
			\Sigma_2'
				& = \int_y^x B(x/t) d\tilde{S}_{\chi}(t)  = -x \int_y^x 	\tilde{S}_{\chi}(t) B'(x/t) t^{-2} dt.
	\end{aligned}
	\end{equation*}
	From \eqref{eqn:B_derivative}, it follows that
	\[
	|\Sigma_2'| \ll \int_y^x \frac{|\tilde{S}_{\chi}(t)|}{t} dt.
	\]
	Substituting the identity
	\[
	\tilde{S}_{\chi}(t) = \int_y^t \frac{1}{u} dS_{\chi}(u) = \frac{S_{\chi}(t)}{t} + \int_y^t \frac{S_{\chi}(u)}{u^2} du
	\] 
	into the previous estimate, 	we have that
	\begin{equation*}
	\begin{aligned}
		\Sigma_2' 
			& \ll \int_y^x \frac{|S_{\chi}(t)|}{t^2} dt + \int_y^x \frac{1}{t} \int_y^t \frac{|S_{\chi}(u)|}{u^2} du dt	 \\
			& \ll \int_y^x \frac{|S_{\chi}(t)|}{t^2} dt  + \Big(\int_y^x \frac{1}{t} dt \Big) \Big(\int_y^x \frac{|S_{\chi}(u)|}{u^2} du \Big) \\
			& \ll \log x 	\int_y^x \frac{|S_{\chi}(t)|}{t^2} dt. 
	\end{aligned}
	\label{eqn:log_weighted_partial_sums_Sigma_2}
	\end{equation*}
	Incorporating \eqref{eqn:log_weighted_partial_sums_Sigma_1}, \eqref{eqn:approximation_L1chi}, and the above into \eqref{eqn:log_weighted_partial_sums_Sigma} establishes \eqref{eqn:log_weighted_Dirichlet_sum}. 
\end{proof}
%

%
%
%
\section{Uniform bounds for quadratic characters}
\label{sec:uniform_bounds}
Here we collect known uniform bounds for character sums and values of the logarithmic derivatives of Dirichlet $L$-functions for quadratic characters. We apply the former to obtain estimates for the error terms arising in \cref{lemma:weighted_Dirichlet_sum}. 
%
%
%
\subsection{Character sums}
\label{subsec:character_sums_uniform}
We state the celebrated result of Burgess \cite{Burgess-1963} specialized to quadratic characters  which was extended from cube-free moduli to all moduli by Heath-Brown \cite[Lemma 2.4]{Heath-Brown-1992}. 

%
%
\begin{lemma}
	\label{lemma:Burgess}
	Let $\chi \pmod{D}$ be a quadratic Dirichlet character. For any $\eta > 0, N \geq 1,$ and integer $k \geq 3$,
	\[
	\sum_{n \leq N} \chi(n) \ll_{\eta,k} D^{(k+1)/4k^2+\eta} N^{1-1/k}. 
	\]
\end{lemma}

We also record the well known GRH-conditional estimate for character sums. 

%
%
\begin{lemma}
	\label{lemma:GRH}
	Let $\chi \pmod{D}$ be a non-principal Dirichlet character and suppose $L(s,\chi)$ satisfies GRH. For any $\eta > 0$ and $N \geq 1,$
	\[
	\sum_{n \leq N} \chi(n) \ll_{\eta} D^{\eta} N^{1/2}. 
	\]
\end{lemma}

%
%
%
\subsection{Errors from \cref{lemma:weighted_Dirichlet_sum}}
The results of \cref{subsec:character_sums_uniform} allow us to obtain power-saving estimates for the error terms in \cref{lemma:weighted_Dirichlet_sum} for small values of $x$ relative to the conductor $D$. 
%
%
%
\begin{lemma}
	\label{lemma:uniform_bounds}
	Let $\chi \pmod{D}$ be a quadratic Dirichlet character and let $0 < \epsilon < \tfrac{1}{20}$ be arbitrary. Let $\cE_0(x;\chi)$ and $\cE_1(x;\chi)$ be as in \cref{lemma:weighted_Dirichlet_sum}.   If $x \gg_{\epsilon} D^{1/4+\epsilon}$ then 
	\begin{equation}
	\cE_0(x; \chi) \ll_{\epsilon} x^{1-\epsilon^2/2} \qquad \cE_1(x; \chi) \ll_{\epsilon} x^{-\epsilon^2/2}. 
	\label{eqn:uniform_bounds_unconditional}
	\end{equation}
	If $L(s,\chi)$ satisfies GRH and $x \gg_{\epsilon} D^{\epsilon}$ then
	\begin{equation}
		\cE_0(x; \chi) \ll_{\epsilon} x^{2/3+\epsilon} \quad \text{and} \quad \cE_1(x;\chi) \ll_{\epsilon} x^{-1/3+\epsilon}. 
		\label{eqn:uniform_bounds_GRH}
	\end{equation}
\end{lemma}

\begin{proof} First, we consider $\cE_0 = \cE_0(x; \chi)$. Let $1 \leq y \leq x$ be a parameter yet to be chosen. From  \cref{lemma:Burgess} and the definition of $\cE_0$, we have that
	\[
	\mathcal{E}_0 \ll_{\eta,k}  x^{-1} y^2  +  D^{\frac{k+1}{4k^2} + \eta} x^{2-1/k} y^{-1}. 
	\]
	Selecting
	\[
	y = D^{\frac{k+1}{12k^2} + \frac{\eta}{3}} x^{1-1/3k}, \qquad \eta = \frac{1}{8k^2}, \qquad k = \lceil 1/\epsilon \rceil \geq 20
	\]
	yields the estimate for $\cE_0$ in \eqref{eqn:uniform_bounds_unconditional} since $x \gg_{\epsilon} D^{1/4+\epsilon}$.  Now, assume GRH holds for $L(s,\chi)$ and $x \gg_{\epsilon} D^{\epsilon}$.  Utilize \cref{lemma:GRH} with $\eta = 3\epsilon^2/2$ and select $y = D^{\epsilon^2/2}x^{5/6}$. This implies that
	\[
	\mathcal{E}_0 \ll_{\epsilon}  \frac{y^2}{x} + \frac{D^{3\epsilon^2/2}  x^{3/2}}{y} \ll_{\epsilon} D^{\epsilon^2} x^{2/3} \ll_{\epsilon} x^{2/3 + \epsilon}
	\]
	as desired. The arguments for $\cE_1 = \cE_1(x;\chi)$ are similar. Again, by \cref{lemma:Burgess}, 
	\[
	\cE_1 \ll_{k,\eta} \frac{y}{x} + \frac{D^{\frac{k+1}{4k^2} + \eta} x^{1-1/k} (\log x)^2}{y}.	
	\]
	Selecting
	\[
	y = D^{\frac{k+1}{8k^2} + \frac{\eta}{2}} x^{1- 1/2k} , \qquad \eta = \frac{1}{8k^2}, \qquad k = \lceil 1/\epsilon \rceil \geq 20 
	\]
	yields the estimate for $\cE_1$ in \eqref{eqn:uniform_bounds_unconditional} since $x \gg_{\epsilon} D^{1/4+\epsilon}$. If $L(s,\chi)$ satisfies GRH then we use \cref{lemma:GRH} with $\eta = \epsilon^2$ and select $y = D^{\epsilon^2} x^{2/3}$. This implies that
	\[
	\cE_1 \ll_{\epsilon}  \frac{y}{x} + D^{3\epsilon^2/2} y^{-1/2} (\log x)^2 \ll_{\epsilon} D^{\epsilon^2/2} x^{-1/3} (\log x)^2 \ll_{\epsilon} x^{-1/3 + \epsilon}
	\]
	as desired. 
\end{proof}

%
%
%
\subsection{Logarithmic derivatives}

We record two results related to the value of the logarithmic derivative of $L(s,\chi)$ at $s=1$ for a quadratic Dirichlet character $\chi$. 

%
%
%
\begin{lemma}[Heath-Brown]
If $\chi \pmod{D}$ is a quadratic character  then, for $\eta > 0$,
\[
-\frac{L'}{L}(1,\chi) \leq (\frac{1}{8} + \eta) \log D + O_{\eta}(1).
\] 
\label{lemma:log_derivative_unconditional}	
\end{lemma}
%
\begin{proof}
	This follows from the proof of \cite[Lemma 3.1]{Heath-Brown-1992}  with some minor modifications to allow for any modulus $D$ (not just sufficiently large) and a slightly wider range of the quantity $\sigma$ therein, say $1 < \sigma < 1 + \epsilon$. See \cite[Proposition 2.6]{ThornerZaman-2017} for details.   
\end{proof}
%

%
%
%

\begin{lemma} 
If $\chi \pmod{D}$ is a quadratic character and $L(s,\chi)$ satisfies GRH then for $\eta > 0$
\[
-\frac{L'}{L}(1,\chi) \ll \log\log D \leq \eta \log D + O_{\eta}(1).
\]
\label{lemma:log_derivative_GRH}	
\end{lemma}
%
\begin{proof} This is well known and can be deduced from the arguments in \cref{lemma:log_derivative_fundamental_average}. 
\end{proof}
%

%
%
\section{Average bounds for quadratic characters}
\label{sec:average_bounds}
The purpose of this section is analogous to \cref{sec:uniform_bounds} except we focus on estimates averaging over a certain class of quadratic characters attached to discriminants. To be more specific, for $Q \geq 3$, recall
\begin{equation*}
	\begin{aligned}
		\mathfrak{D}(Q)
			& = \{ \text{discriminants $-D$  with $3 \leq D \leq Q$ }\}. 
	\end{aligned}
\end{equation*}
Here a discriminant is that of a primitive positive definite binary quadratic form. 
We emphasize that a discriminant $-D \in \mathfrak{D}(Q)$ is not necessarily fundamental. The associated Kronecker symbol $\chi_{-D}(\, \cdot \,) = (\frac{-D}{\, \cdot \,})$ is itself a quadratic Dirichlet character. Note the character is primitive if and only if $-D$ is a fundamental discriminant. Our goal is to average certain quantities involving $\chi_{-D}$ over $-D \in \kD(Q)$. 

%
%
%
\subsection{Character sums}
\label{subsec:character_sums_average}

We record a special case of Heath-Brown's mean value theorem for primitive quadratic characters \cite[Corollary 3]{Heath-Brown-1995}. 

%
%
%
\begin{lemma}[Heath-Brown]
	\label{lemma:mean_value_quadratic}
	Let $N,Q \geq 1$ and let $a_1,\dots,a_n$ be arbitrary complex numbers. Let $\mathcal{S}(Q)$ denote the set of all primitive quadratic characters of conductor at most $Q$. Then
	\[
	\sum_{\chi \in \mathcal{S}(Q)} \Big| \sum_{n \leq N} a_n \chi(n) \Big|^2 \ll_{\eta} ( (QN)^{1+\eta} +  Q^{\eta} N^{2+\eta} ) \max_{1 \leq n \leq N} |a_n|
	\]
	for any $\eta > 0$. 	
\end{lemma}
%
  Using \cref{lemma:mean_value_quadratic}, we deduce an analogous mean value result for the quadratic characters attached to such discriminants. 

%
%
%
\begin{lemma}
	\label{lemma:mean_value_quadratic_bqfs}
	Let $N,Q \geq 1$ and  let $\mathfrak{D}(Q)$  be defined by \eqref{def:discriminants}. Then
	\[
	\sum_{\substack{ -D \in \mathfrak{D}(Q) } }\Big| \sum_{n \leq N} \chi_{-D}(n) \Big|^2 \ll_{\eta} (QN)^{1+\eta} +  Q^{1/2+\eta} N^{2+\eta} 
	\]
	for any $\eta > 0$. 	
\end{lemma}
%

%
\begin{proof} If $-D \in \mathfrak{D}(Q)$ then $-D = \Delta k^2$ where $\Delta \in \Z$ is the discriminant of some imaginary quadratic field and $k \geq 1$ is some integer. Consequently, the Kronecker symbol $\chi_{-D}( \, \cdot \,) = (\frac{-D}{\, \cdot \,})$ is induced by the \emph{primitive} quadratic character $\chi_{\Delta}(\, \cdot \,) = (\frac{\Delta}{\, \cdot \,})$. Moreover,
\[
\chi_{-D}(n) = \chi_{k^2}(n) \chi_{\Delta}(n). 
\]
For the details on these facts, see \cite[\textsection 7]{Cox-2013}.  Therefore, by \cref{lemma:mean_value_quadratic}, 
\begin{equation*}
\begin{aligned}
	\sum_{-D \in \mathfrak{D}(Q)}\Big| \sum_{n \leq N} \chi_{-D}(n) \Big|^2
		& \leq \sum_{1 \leq k \leq \sqrt{Q}} \sum_{\substack{1 \leq |\Delta| \leq Q/k^2 \\ \chi_{\Delta} \text{ primitive}}} \Big| \sum_{n \leq N} \chi_{k^2}(n) \chi_{\Delta}(n) \Big|^2 \\
		& \ll_{\eta} \sum_{1 \leq k \leq \sqrt{Q}}  \big( (QN)^{1+\eta} k^{-2-2\eta}  + Q^{\eta} k^{-2\eta} N^{2+\eta} \big) \\
		& \ll_{\eta} (QN)^{1+\eta} + Q^{1/2+\eta} N^{2+\eta}
\end{aligned}	
\end{equation*}
as desired.
\end{proof}

%
%
%
\subsection{Error terms from \cref{lemma:weighted_Dirichlet_sum}}
Again, the results in \cref{subsec:character_sums_average} leads to estimates for the error terms in \cref{lemma:weighted_Dirichlet_sum}. 

%
%
%
\begin{lemma}
	\label{lemma:charsum_errors_dyadic_average}
	Let $X \geq 1, Q \geq 3$ and  let $\mathfrak{D}(Q)$  be defined by \eqref{def:discriminants}. For $x \geq 1$ and any quadratic character $\chi$, define $\cE_0(x; \chi)$ and $\cE_1(x; \chi)$ as in \cref{lemma:weighted_Dirichlet_sum}. Then
	\begin{equation}
		\begin{aligned}
				\sum_{-D \in \mathfrak{D}(Q) } \max_{X \leq x \leq 2X} \cE_0(x; \chi_{-D} ) 
					& \ll_{\eta} 	 Q^{1+\eta} X^{3/5+\eta} + Q^{3/4+\eta} X^{1+\eta}   \\
				\sum_{-D \in \mathfrak{D}(Q) } \max_{X \leq x \leq 2X} \cE_1(x; \chi_{-D} ) 
					& \ll_{\eta} 	 Q^{1+\eta} X^{-1/3+\eta} + Q^{3/4+\eta} X^{\eta}  \\
		\end{aligned}
	\end{equation}
	for $\eta > 0$. 
\end{lemma}
%
\begin{proof}
	For $X \leq x \leq 2X$, let $1 \leq y \leq X$ be an unspecified parameter, depending only on $X$. Set $y_0 :=  y Q^{1/2+\eta}$.  By Polya-Vinogradov, $|S_{\chi}(t)| \ll Q^{1/2} \log Q$ for any character $\chi \pmod{q}$ with $q \leq Q$. Therefore, for such $\chi$, 
	\begin{equation*}
	\begin{aligned}
		\cE_0(x; \chi) 
			& \ll \frac{y^2}{x} + |S_{\chi}(y)| + x \int_y^{\infty} \frac{|S_{\chi}(t)|}{t^2} dt \\
			& \ll_{\eta} \frac{y^2}{x} + |S_{\chi}(y)| + x \int_y^{y_0} \frac{|S_{\chi}(t)|}{t^2} dt + \frac{x}{y}. 		\\
	\end{aligned}
	\end{equation*}
	As $y$ and $y_0$ depend only on $X$ and $Q$, it follows that
	\[
	\max_{X \leq x \leq 2X} \cE_0(x; \chi) \ll_{\eta} \frac{y^2}{X} + |S_{\chi}(y)| + X \int_y^{y_0} \frac{|S_{\chi}(t)|}{t^2} dt + \frac{X}{y}. 
	\]
	Summing the above expression over $\chi = \chi_{-D}$ with $-D \in \mathfrak{D}(Q)$, applying Cauchy-Schwarz, and invoking the previous lemma, we see that
	\begin{equation*}
	\begin{aligned}
			\sum_{-D \in \mathfrak{D}(Q) } \max_{X \leq x \leq 2X} \cE_0(x; \chi_{-D} ) 
				& \ll_{\eta} \frac{Qy^2}{X} + Q^{1+\eta} y^{1/2+\eta} + Q^{3/4+\eta} y^{1+\eta} + \frac{X Q}{y}  \\
				& \qquad + X Q^{1/2} \int_y^{y_0} \frac{ (Qt)^{1/2+\eta} + Q^{1/4+\eta/2} t^{1+\eta/2} }{t^2} dt \\
				& \ll_{\eta}  \frac{Qy^2}{X} + Q^{1+\eta} y^{1/2+\eta} + Q^{3/4+\eta} y^{1+\eta}  + \frac{X Q}{y}  \\
				& \qquad + X Q^{1+\eta} y^{-1/2+\eta}  + X Q^{3/4+\eta/2} y_0^{\eta/2}  \\	
				& \ll_{\eta}  \frac{Qy^2}{X} +  \frac{X Q}{y} + X Q^{1+\eta} y^{-1/2+\eta}  + X^{1+\eta} Q^{3/4+\eta}. \\	
	\end{aligned}
	\end{equation*}
	In the last step, we used the definition of $y_0$ and the fact that $y \leq X$. 
	Selecting $y = X^{4/5}$ implies the desired result after rescaling $\eta$ if necessary. We follow the same procedure for the average of $\cE_1$. First, we deduce that
	\[
	\max_{X \leq x \leq 2X} \cE_1(x; \chi_{-D} )  
		\ll_{\eta} \frac{y}{X} + \log X \int_y^{y_0} \frac{|S_{\chi}(t)| \log t}{t^2}dt + \frac{\log(Qy)}{y}
	\]
	Again, summing over $\chi = \chi_{-D}$ with $-D \in \mathfrak{D}(Q)$, we similarly conclude that
	\begin{equation*}
	\begin{aligned}
			\sum_{-D \in \mathfrak{D}(Q) } \max_{X \leq x \leq 2X} \cE_1(x; \chi_{-D} ) 
				& \ll_{\eta} \frac{Qy}{X} + \frac{Q \log(Qy)}{y} + Q^{1/2} \log X \int_y^{y_0} \frac{ (Qt)^{1/2+\eta} + Q^{1/4+\eta/2} t^{1+\eta/2} }{t^2} dt \\
				& \ll_{\eta}  \frac{Qy}{X} + \frac{Q \log(Qy)}{y} + ( Q^{1+\eta} y^{-1/2+\eta}  + Q^{3/4+\eta/2} y_0^{\eta/2}) \log X  \\	
				& \ll_{\eta}  \frac{Qy}{X} +  \frac{Q^{1+\eta}}{y^{1-\eta}} + X^{\eta} Q^{1+\eta} y^{-1/2+\eta}  + X^{\eta} Q^{3/4+\eta}. 
	\end{aligned}
	\end{equation*}
	Selecting $y = X^{2/3}$ yields the desired result. 
	
\end{proof}
\begin{lemma}
	\label{lemma:charsum_errors_average}
	Let $0 < \epsilon < 1/8, X \geq 1$ and $Q \geq 3$. Let $c = c(\epsilon) > 0$ and $C = C(\epsilon) \geq 1$ be arbitrary constants. For all except at most $O_{\epsilon}(Q^{1-\epsilon/10})$ discriminants $-D \in \mathfrak{D}(Q)$, 
	\begin{equation}
	\cE_0(x; \chi_{-D}) \leq x^{7/8+\epsilon}, \qquad \cE_1(x; \chi_{-D}) \leq x^{-1/8+\epsilon},
	\label{eqn:charsum_errors_average}
	\end{equation}
	uniformly for $c D^{\epsilon} \leq  x \leq C D^{2+\epsilon}$. Here $\cE_0$ and $\cE_1$ are defined as in \cref{lemma:weighted_Dirichlet_sum}. 
\end{lemma}
%

%
\begin{proof} Without loss, we need only consider discriminants $-D \in \kD(Q)$ satisfying $D \geq Q^{1-\epsilon}$ since the remainder are  a collection of negligible size $O(Q^{1-\epsilon})$. For $X \geq 1$, define 
\[
\mathfrak{E}(X,Q,\epsilon) := \{ -D \in \kD(Q) : \text{ there exists $X \leq x \leq 2X$ violating \eqref{eqn:charsum_errors_average} for $\chi_{-D}$ } \}.
\]
By \cref{lemma:charsum_errors_dyadic_average} with $\eta=\epsilon/8$, it follows that
\begin{equation*}
	\begin{aligned}
		\mathfrak{E}(X,Q,\epsilon) 
				& \ll_{\epsilon} Q^{1+\epsilon/8} X^{-11/40-7\epsilon/8} + Q^{3/4+\epsilon/8} X^{1/8-7\epsilon/8}. \\
	\end{aligned}
\end{equation*}
Dyadically summing this estimate over $X$ between $cQ^{\epsilon(1-\epsilon)}$ and $CQ^{2+\epsilon}$, we see that the total number of discriminants $-D \in \mathfrak{D}(Q)$ satisfying $D \geq Q^{1-\epsilon}$ and violating \eqref{eqn:charsum_errors_average} anywhere in the range $cD^{\epsilon} \leq x \leq C D^{2+\epsilon}$ is bounded by
\[
\ll_{\epsilon} 
	(Q^{1+\frac{1}{8}\epsilon -\frac{11}{40}\epsilon(1-\epsilon)} + Q^{1 -\frac{3}{2}\epsilon}) \log Q 
\ll_{\epsilon} Q^{1- \frac{37}{320}\epsilon}\log Q 
\ll_{\epsilon} Q^{1-\frac{1}{10}\epsilon}
\]
as $\epsilon < 1/8$.  
\end{proof}

%
%
%
\subsection{Logarithmic derivatives}
For $Q \geq 3$, define
\[
\kD^*(Q) = \{ \text{fundamental discriminants $\Delta$ with $3 \leq |\Delta| \leq Q$} \}. 
\]
We show that, aside from a sparse set of fundamental discriminants in $\kD^*(Q)$, the logarithmic derivative of $L(s,\chi_{\Delta})$ at $s=1$ satisfies a GRH-quality bound. The key inputs are the explicit formula and Jutila's zero density estimate for primitive quadratic characters. 

%
%
%
%
\begin{lemma} 
	\label{lemma:log_derivative_fundamental_average}
	Let $Q \geq 3$ and $\epsilon > 0$ be arbitrary. For all except at most $O_{\epsilon}( Q^{3/4+\epsilon})$ fundamental discriminants $\Delta \in \kD^*(Q)$, 
	\begin{equation}
	-\frac{L'}{L}(1,\chi_{\Delta}) \ll_{\epsilon} \log\log|\Delta|. 
	\label{eqn:log_derivative_fundamental_average}
	\end{equation}
\end{lemma}
%

%
\begin{proof} We modify the arguments leading to \cite[Theorem 3]{MourtadaKumar-Murty-2013}. Define $\kD_{\epsilon}^*(Q)$ to be the set of fundamental discriminants $\Delta \in \kD^*(Q)$ such that $|\Delta| \leq Q$ and whose $L$-function $L(s,\chi_{\Delta})$ is zero-free in the rectangle
\begin{equation}
	\frac{1}{2} < \Re\{s\} < 1 \qquad |\Im\{s\}| \leq |\Delta|^{\epsilon}. 
	\label{eqn:zero-free_rectangle}
\end{equation}
First, we estimate $-\frac{L'}{L}(s,\chi_{\Delta})$ for $\Delta \in \kD_{\epsilon}^*(Q)$. For simplicity, write $\chi = \chi_{\Delta}$. From the explicit formula in the form given by \cite[p. 261]{IharaMurtyShimura-2009}, one can verify that
	\begin{equation}
	\begin{aligned}
		-\frac{L'}{L}(1,\chi) 
			& = \frac{1}{y-1} \sum_{m <y} \Big(\frac{y}{m}-1 \Big) \Lambda(m) \chi(m)  - \frac{1}{y-1} \sum_{\rho} \frac{y^{\rho}-1}{\rho (1-\rho)} + O\Big(\frac{\log y}{y}\Big) \\
	\end{aligned}
	\end{equation}
	for $y \geq 2$, where the sum is taken over all non-trivial zeros $\rho = \beta + i\gamma$ of $L(s,\chi)$.   From \cite[Equation  (5.4.6)]{IharaMurtyShimura-2009} and the prime number theorem, it follows for $T \geq 1$ that
	\[
	-\frac{L'}{L}(1,\chi) = -\frac{1}{y-1} \sum_{\substack{\rho \\ |\gamma| \leq T}} \frac{y^{\rho}-1}{\rho (1-\rho)} + O\Big(\log y + \frac{\log(|\Delta|T)}{T} + \frac{\log^2 |\Delta|}{y}\Big). 
	\]
	Set $T = |\Delta|^{\epsilon}$. By the symmetry of the functional equation for real characters $\chi$ and the fact that $L(s,\chi)$ has no zeros in \eqref{eqn:zero-free_rectangle}, it follows that every zero appearing in the sum over $\rho$ satisfies $\Re\{\rho\} = \frac{1}{2}$. Thus, trivially bounding the remaining zeros, we deduce that
	\begin{equation*}
	\begin{aligned}
	-\frac{L'}{L}(1,\chi) 
		& \ll \log y + y^{-1/2} \sum_{\substack{ \rho = \frac{1}{2}+i\gamma \\ |\gamma| \leq |\Delta|^{\epsilon} } } \frac{1}{1+|\gamma|^2}  + \frac{\log |\Delta|}{|\Delta|^{\epsilon}} + \frac{\log^2 |\Delta|}{y} \\
		& \ll \log y + \frac{\log |\Delta|}{y^{1/2}} + \frac{\log |\Delta|}{|\Delta|^{\epsilon}} + \frac{\log^2 |\Delta|}{y} 
		\end{aligned}
	\end{equation*}
	for $y \geq 2$. Setting $y = (\log |\Delta|)^2 + 2$ implies \eqref{eqn:log_derivative_fundamental_average} holds for all $\Delta \in \kD_{\epsilon}^*(Q)$. 
	
	It remains to show that  the number of  discriminants $\Delta \in \kD^*(Q) \setminus \kD_{\epsilon}^*(Q)$ is small. Jutila's zero density estimate \cite[Theorem 2]{Jutila-1975} implies that
	\[
	\sum_{\substack{\Delta \in \kD^*(Q)}} N(\sigma,T,\chi_{\Delta}) \ll_{\epsilon} (QT)^{\tfrac{7-6\sigma}{6-4\sigma} + \frac{\epsilon}{4}}, 
	\]
	where $N(\sigma,T,\chi)$ is the number of zeros $\rho = \beta+i\gamma$ of $L(s,\chi)$ with $\sigma < \beta < 1$ and $|\gamma| \leq T$. Setting $\sigma = 1/2$ and $T = Q^{\epsilon}$, we see that the number of fundamental discriminants $\Delta \in \kD^*(Q)$ whose $L$-function $L(s,\chi_{\Delta})$ has a zero in the rectangle \eqref{eqn:zero-free_rectangle} is at most $O_{\epsilon}(Q^{3/4+\epsilon})$. Hence, $|\kD^*(Q) \setminus \kD_{\epsilon}^*(Q)| \ll_{\epsilon} Q^{3/4+\epsilon}$ as required. 
\end{proof}
%

\cref{lemma:log_derivative_fundamental_average} implies the same type of result for the set of all discriminants $\kD(Q)$.

%
%
%
%
\begin{lemma}
		\label{lemma:log_derivative_average}
		Let $Q \geq 3$ and $\epsilon > 0$ be arbitrary. For all except at most $O_{\epsilon}( Q^{3/4+\epsilon})$   discriminants $-D \in \kD(Q)$, 
	\[
	-\frac{L'}{L}(1,\chi_{-D}) \ll_{\epsilon} \log\log D \leq \epsilon \log D + O_{\epsilon}(1). 
	\]
\end{lemma}
%

%
\begin{proof}
	Let $-D \in \kD(Q)$ so, as in the proof of \cref{lemma:mean_value_quadratic_bqfs}, we may write $-D = \Delta k^2$ for some negative fundamental discriminant $\Delta$ and an integer $k \geq 1$. It follows that $\chi_{-D}$ is induced by the primitive character $\chi_{\Delta}$ and, in particular, $\chi_{-D} = \chi_{\Delta} \chi_{k^2}$. This implies that
	\[
	\Big| \frac{L'}{L}(1,\chi_{-D}) - \frac{L'}{L}(1,\chi_{\Delta})\Big| \leq \sum_{ (n,k) \neq 1} \frac{\Lambda(n)}{n} \ll \sum_{p \mid k} \frac{\log p}{p} \ll \log\log k. 
	\]
	Thus, if $\Delta$ is a fundamental discriminant satisfying \eqref{eqn:log_derivative_fundamental_average} then
	\[
	-\frac{L'}{L}(1,\chi_{-D}) \ll_{\epsilon} \log\log |\Delta| + \log\log k \ll_{\epsilon} \log\log D.
	\]
	 \cref{lemma:log_derivative_fundamental_average} implies that the total number of discriminants $-D$ failing the above bound is 
	\[
	\ll_{\epsilon} \sum_{k \leq \sqrt{Q}} \Big(\frac{Q}{k^2}\Big)^{3/4+\epsilon} \ll_{\epsilon} Q^{3/4+\epsilon}. 
	\]
	This completes the proof. 
\end{proof}

%
%
\section{Congruence sum decomposition}
\label{sec:congruence_sums}
Let $f(u,v) = au^2 + buv + cv^2$ be a form with discriminant $-D$. For this section, we will not require $f$ to be primitive. For any integer $n \geq 0$, define
%
%
%
\begin{equation}
	r_f(n) := |\{ (u,v) \in \Z^2 : n = f(u,v) \}|.
\end{equation}
Moreover, for $x \geq 1$ and positive integers $\ell$ and $d$, define 
\begin{equation}
	\begin{aligned}
	\cA = \cA(x, f) & := \{ (u,v) \in \Z^2 : f(u,v) \leq x \}, \\
	\cA_{\ell} = \cA_{\ell}(x, f) & := \{ (u,v) \in \cA : f(u,v) \equiv 0 \pmod{\ell} \}, \\
	\cA_{\ell}(d) = \cA_{\ell}(x, f; d) & := \{ (u,v) \in \cA_{\ell} : (v,\ell) = d\}. 
	\end{aligned}
	\label{def:sieve_sequence}
\end{equation}
%
We will suppress the dependence on $x$ and $f$ whenever it is clear from context. This will be the case for almost the entirety of the paper. Observe that
\begin{equation}
|\cA| = \sum_{n \leq x} r_f(n)
\quad \text{and} \quad
|\cA_{\ell}| = \sum_{\substack{n \leq x \\ \ell \mid n}} r_f(n) = \sum_{d \mid \ell} |\cA_{\ell}(d)|. 
\label{eqn:sieve_sequence_observations}
\end{equation}
Note the last identity holds since $\cA_{\ell}$ is a disjoint union of the sets $\cA_{\ell}(d)$ over $d \mid \ell$. To calculate $|\cA_{\ell}(d)|$, and subsequently $|\cA_{\ell}|$, we will need to decompose it into sums similar to $|\cA_{\ell}(1)|$ and estimate them with uniformity over all parameters. To this end, we introduce some additional notation. For any integer $\ell \geq 1$ and $m \in \Z/\ell \Z$, define 
%
%
%
\begin{equation}
	\begin{aligned}
	\cB_{\ell} = \cB_{\ell}(x, f) & := \{ (u,v) \in \cA : (v,\ell)=1, f(u,v) \equiv 0 \pmod{\ell} \}, \\
	\cB_{\ell}(m) = \cB_{\ell}(x, f; m) & := \{ (u,v) \in \cA : (v,\ell) = 1, \, u \equiv mv \pmod{\ell} \}. 
	\end{aligned}
	\label{def:sieve_sequence_primitive}
\end{equation}
%
Note that $\cB_{\ell}$ is exactly the same as $\cA_{\ell}(1)$, but we distinguish it for the sake of clarity. The crucial property of the sets $\cB_{\ell}$ and $\cB_{\ell}(m)$ is summarized in the following lemma.

%
%
%
\begin{lemma}
	\label{lemma:decompose_primitive_congruence_sum}
	Let $f(u,v) = au^2 + buv + cv^2$ be a positive definite binary integral quadratic form of discriminant $-D$ and let $\ell \geq 1$ be a squarefree integer. Define
	\[
	\cM(\ell) = \cM_f(\ell) := \{ m \in \Z/\ell \Z : am^2 + bm + c \equiv 0 \pmod{\ell} \}. 
	\]
	Then
	\begin{equation}
	|\cB_{\ell}| = \sum_{m \in \cM(\ell)} |\cB_{\ell}(m)|.
	\label{eqn:decompose_primitive_congruence_sum}
	\end{equation}
	Furthermore, $M(\ell) = M_f(\ell) := |\cM(\ell)|$ is a non-negative multiplicative function  satisfying 
	\begin{equation}
	\label{def:M_solns}
	M(p) = \begin{cases} 
 					1 + \chi(p) & \text{if } p \nmid a, \\
 					\chi(p) & \text{if } p \mid a \text{ and } p \nmid (a,b,c), \\
 					p & \text{if } p \mid (a,b,c), 
			\end{cases}		
	\end{equation}
	for all primes $p$. Here $\chi = \chi_{-D}$ is the corresponding Kronecker symbol.
\end{lemma}
%

\begin{proof}
Let $(u,v) \in \cB_{\ell}$. As $(v,\ell) = 1$, select $m \in \Z/\ell\Z$ such that $u \equiv mv \pmod{\ell}$. Thus,
\[
f(u,v) \equiv 0 \pmod{\ell} \iff (am^2 + bm + c)v^2 \equiv 0 \pmod{\ell} \iff m \in \cM(\ell). 
\]
This implies $\cB_{\ell}$ is a union of $\cB_{\ell}(m)$ over $m \in \cM(\ell)$. One can verify from \eqref{def:sieve_sequence_primitive} that $m_1 \not\equiv m_2 \pmod{\ell}$ implies $\cB_{\ell}(m_1) \cap \cB_{\ell}(m_2) = \emptyset$. Thus, the union is in fact disjoint yielding \eqref{eqn:decompose_primitive_congruence_sum}.

Next, we count $M(\ell) = |\cM(\ell)|$. The function $M(\ell)$ is multiplicative by the Chinese Remainder Theorem.  Let $p$ be an odd prime. If $p \nmid a$ then $M(p) = 1+\chi(p)$ by the definition of $\chi$. If $p \mid a$ then for $m \in \cM(\ell)$
\begin{equation}
0 \equiv am^2 +bm + c \equiv bm+c \pmod{p}.
\label{eqn:quadratic_to_linear_congruence}
\end{equation}
Note in this scenario $\chi(p) = 0$ or $1$ only. We consider cases.
\begin{itemize}
	\item If $p \nmid b$ then $m \equiv - b^{-1} c \pmod{p}$ is the only solution to \eqref{eqn:quadratic_to_linear_congruence}. Thus, $M(p) = 1 = \big(\frac{b^2-4ac}{p}\big) = \chi(p)$. If $p=2$ then note $b^2-4ac \equiv 1 \pmod{8}$, so $\chi(2) = 1$ indeed. 
	\item If $p \mid b$ then condition \eqref{eqn:quadratic_to_linear_congruence} becomes $c \equiv  0 \pmod{p}$. We further subdivide the cases.
		\begin{itemize}
				\item If $p \nmid c$ then no value of $m$ satisfies \eqref{eqn:quadratic_to_linear_congruence} implying $M(p) = 0 = (\frac{b^2-4ac}{p}) = \chi(p)$. 
				\item If $p \mid c$ then $p \mid (a,b,c)$ in this subcase. Hence, all $m \in \Z/p\Z$ vacuously satisfy \eqref{eqn:quadratic_to_linear_congruence} so $M(p) = p$. 
		\end{itemize}
\end{itemize}
Comparing these cases, we see $M(p)$ indeed satisfies \eqref{def:M_solns} for all odd primes $p$. For $p=2$,  one can verify by a tedious case analysis that $M(2)$ also satisfies \eqref{def:M_solns}.  
\end{proof}
%

In light of \cref{lemma:decompose_primitive_congruence_sum}, the main goal of this section is to determine the size of $|\cB_{\ell}(m)|$ for any $m \in \Z/\ell \Z$.  For convenience, set
%
%
\begin{equation}
	V = V(x, f) := \sqrt{\frac{4ax}{D}}. 
	\label{def:V}
\end{equation}
%
This notation will be used throughout the paper. 
While we are more interested when $V \geq 1$, we only assume $x \geq 1$ in all of our arguments so it is possible that $0 < V < 1$. Recall $\varphi$ denotes the Euler totient function and $\tau$ is the divisor function. 

%
%
%
\begin{lemma}
	\label{lemma:primitve_congruence_sum_fibre} Let $f(u,v) = au^2 + buv + cv^2$ be a positive definite binary integral quadratic form of discriminant $-D$. Let $\ell \geq 1$ be a squarefree integer and $m \in \Z/\ell \Z$. For $x \geq 1$, 
	\begin{equation}
	|\cB_{\ell}(m)| = \frac{\varphi(\ell)}{\ell^2} \cdot \frac{\pi \sqrt{D}}{2a} V^2 + O\Big(V +  \ell^{1/2} \tau(\ell) \frac{\sqrt{D}}{a} V^{1/2} + \delta(\ell)\Big),
	\label{eqn:congruence_sum_fibre}
	\end{equation}
	where $\cB_{\ell}(m) = \cB_{\ell}(x, f; m)$ is defined by \eqref{def:sieve_sequence_primitive},  $V= V(x,f)$ is defined by \eqref{def:V}, and $\delta(\ell)$ is the indicator function for $\ell=1$.
\end{lemma}
%
\begin{remark}
	We emphasize that the righthand side of \eqref{eqn:congruence_sum_fibre} is independent $m \in \Z/\ell\Z$. 	
\end{remark}

\begin{proof}
Counting the numbers of pairs $(u,v) \in \Z^2$ satisfying $f(u,v) \leq x$ amounts to verifying the inequality
\[
(2au + bv)^2 + D v^2 \leq 4ax. 
\]
Fixing $v$, any $u$ satisfying the above inequality lies in the range
\[
\frac{-bv -\sqrt{4ax - D v^2}}{2a} \leq u \leq \frac{-bv+\sqrt{4ax-Dv^2}}{2a}.
\]
Without loss, we may assume $m$ is an integer lying in $\{0,1,\dots,\ell-1\}$. Restricting to $u  = mv + j\ell$, we see that the integer $j$ must lie in the range
\[
\frac{- (b+2am)v-\sqrt{4ax - D v^2} }{2a\ell} \leq j \leq \frac{- (b+2am)v + \sqrt{4ax-Dv^2}}{2a\ell}.
\]
For each fixed $v$ and solution $u \equiv m v \pmod{\ell}$, the total number of such integers $j$ is therefore
\[
\tfrac{1}{\ell} F(v) + O(1),
\]
where
\begin{equation}
F(v) = \frac{1}{a} \sqrt{4ax-Dv^2} = \frac{\sqrt{D}}{a} \sqrt{V^2 - v^2}.
\label{def:F}	
\end{equation}
Now, summing over all integers $v$ satisfying $|v| \leq V$ and $(v,\ell) = 1$, we deduce that
\begin{equation*}
\begin{aligned}
|\cB_{\ell}(m)| 
	& = \frac{1}{\ell} \sum_{\substack{|v| \leq V \\ (v,\ell)=1}}   F(v) + O(\sum_{\substack{|v| \leq V \\ (v,\ell)=1}} 1 ).
\end{aligned}
\end{equation*}
The term $v=0$ contributes to the above sums if and only if $\ell = 1$. Let $\delta(\ell)$ be the indicator function for $\ell=1$. We separate the term $v=0$, if necessary, in the sums above and note $F(v)$ is even to see that
\begin{equation}
	\begin{aligned}
		|\cB_{\ell}(m)| 
		& = \frac{2}{\ell} \sum_{\substack{1 \leq v \leq V \\ (v,\ell)=1}} F(v) + \frac{\delta(\ell)}{\ell} \frac{\sqrt{D}}{a} V + O(V + \delta(\ell)).
	\end{aligned}
	\label{eqn:fibre_step1}
\end{equation}
We remove the condition $(v,\ell) = 1$ via Mobius inversion and deduce that
\begin{equation}
 \sum_{\substack{1 \leq v \leq V \\ (v,\ell)=1}} F(v) 
 	= \sum_{d \mid \ell} \mu(d)  \sum_{\substack{1 \leq w \leq V/d}} F(dw). 
 \label{eqn:sqrt_average_primitive}
\end{equation}
Set $W_d := V/d$ and notice $F(dw) =  \frac{d\sqrt{D}}{a} \sqrt{W_d^2-w^2}$. By \cref{lemma:sqrt_average}, we see that
\begin{equation*}
\begin{aligned}
\sum_{\substack{1 \leq w \leq V/d}} F(dw) 
& = \frac{d \sqrt{D}}{a} \Big( \frac{\pi V^2}{4d^2} - \frac{V}{2d} + O( \sqrt{V/d}) \Big). \\
\end{aligned}
\end{equation*}
Since $\sum_{d \mid \ell} \frac{\mu(d)}{d} = \frac{\varphi(\ell)}{\ell}$, $\sum_{d \mid \ell} \mu(d) = \delta(\ell)$, and $\sum_{d \mid \ell} d^{1/2} \ll  \ell^{1/2} \tau(\ell)$, it follows by \eqref{eqn:sqrt_average_primitive} that
\begin{equation}
 \sum_{\substack{1 \leq v \leq V \\ (v,\ell)=1}} F(v) = \frac{\varphi(\ell)}{\ell} \frac{\pi \sqrt{D}}{4ad} V^2 - \delta(\ell) \frac{\sqrt{D}}{2a} V + O\Big( \ell^{1/2} \tau(\ell)  \frac{\sqrt{D}}{a} V^{1/2} \Big).
 \label{eqn:sqrt_average_primitive_final}
\end{equation}
Combining \eqref{eqn:fibre_step1} and \eqref{eqn:sqrt_average_primitive_final} yields \eqref{eqn:congruence_sum_fibre}. Note that the terms involving $\delta(\ell)$ cancel.  
\end{proof}
%

We conclude this section by calculating $|\cB_{\ell}|$. 

%
%
%
\begin{lemma}
	\label{lemma:primitive_congruence_sum}
	Let $f(u,v) = au^2 + buv + cv^2$ be a positive definite binary integral quadratic form of discriminant $-D$.  If $\ell \geq 1$ is a squarefree integer then
	\[
	|\cB_{\ell}|  = M(\ell) \Big( \frac{\varphi(\ell)}{\ell^2} \cdot \frac{\pi \sqrt{D}}{2a} V^2 + O\Big(V +  \ell^{1/2} \tau(\ell) \frac{\sqrt{D}}{a} V^{1/2} + \delta(\ell)\Big) \Big),
	\]
	where $\cB_{\ell} = \cB_{\ell}(x, f)$ is defined by \eqref{def:sieve_sequence_primitive},  $V= V(x,f)$ is defined by \eqref{def:V}, $\delta(\ell)$ is the indicator function for $\ell=1$, and $M(\ell) = M_f(\ell)$ is a multiplicative function defined by \eqref{def:M_solns}. 
\end{lemma}
%

\begin{proof}
This is an immediate consequence of \cref{lemma:decompose_primitive_congruence_sum,lemma:primitve_congruence_sum_fibre} since the latter lemma's estimates are uniform over all $m \in \Z/\ell\Z$. 
\end{proof}

%
%
\section{Local densities}
\label{sec:local_densities}
We may now assemble our tools to establish the key technical proposition. Namely, we estimate the congruence sums given by \eqref{eqn:sieve_sequence_observations} and calculate the local densities.  
 
%
%
%
\begin{proposition}
	\label{prop:local_density} Let $f$ be a primitive positive definite binary quadratic form with discriminant $-D$. If $\ell \geq 1$ is a squarefree integer then for any $\epsilon > 0$ and $x \geq 1$,
	\begin{equation}
	|\cA_{\ell}| = \sum_{\substack{n \leq x \\ \ell \mid n}} r_f(n) = 	g(\ell)  \frac{\pi \sqrt{D}}{2a}  V^2 + O\Big( \tau_3(\ell) V + \ell^{1/2} \tau(\ell) \tau_3(\ell) \frac{\sqrt{D}}{a} V^{1/2} + 1 \Big),
	\label{eqn:congruence_sum_final}
	\end{equation}
	where $V = \sqrt{4ax/D}$ and $g$ is a multiplicative function satisfing
	\begin{equation}
		g(p) = \frac{1}{p} \Big(1 + \chi(p) - \frac{\chi(p)}{p} \Big) \quad \text{for all primes $p$}.	
		\label{def:local_density}
	\end{equation}
	Here $\chi = \chi_{-D}$ is the corresponding Kronecker symbol.
\end{proposition}
%

\begin{proof} Let $d \mid \ell$ and let $\cA_{\ell}(d) = \cA_{\ell}(x, f; d)$ be defined by \eqref{def:sieve_sequence}. From observation \eqref{eqn:sieve_sequence_observations}, it suffices to calculate $|\cA_{\ell}(d)|$. First, we introduce some notation. For any integer $r \geq 1$, set 
\begin{equation}
f_r(u,w) := f(u,rw) = au^2 + bruw + cr^2 w^2. 
\label{def:f_r}
\end{equation}
Notice that its discriminant is $-r^2D$. Therefore, it follows for any $\alpha > 0$  that
\[
V(\alpha^2 x,f_r) = \frac{\alpha V}{r} \qquad \text{and} \qquad \chi_{-r^2D}(n) = \begin{cases} \chi(n) & \text{if $(n,r) = 1$} \\
 0 & \text{otherwise,} 	
 \end{cases}
\]
where $V = V(x,f)$ and $\chi = \chi_{-D}$ as usual. 

Now, write $\ell = dk$ so $(d,k) = 1$ as $\ell$ is squarefree. We wish to character each point $(u,v) \in \cA_{\ell}(x, f; d)$. Since $(v,\ell) = d$, it follows by the Chinese Remainder Theorem that
\begin{equation}
\begin{aligned}
f(u,v) \equiv 0 \pmod{\ell} 
	& \iff au^2 \equiv 0 \pmod{d} \quad \text{and} \quad f(u,v) \equiv 0 \pmod{k} \\
	& \iff u \equiv 0 \pmod{\tfrac{d}{(a,d)}} \quad \text{and} \quad f(u,v) \equiv 0 \pmod{k}.  
\end{aligned}
\label{eqn:cov_CRT_1}
\end{equation}
Write $u = \frac{d}{(a,d)} s$ and $v = dt$ for integers $s$ and $t$. Note $(t,k) = 1$ as $(v,\ell) = d$ and $\ell$ is squarefree. Then one can verify that
\begin{equation}
f(u,v) =  \frac{d^2}{(a,d)^2} \cdot \big( a s^2 + b (a,d) st + c(a,d)^2 t^2 \big) = \frac{d^2}{(a,d)^2} \cdot f_{(a,d)}(s,t). 
\label{eqn:cov_CRT_2}
\end{equation}
From this change of variables, \eqref{eqn:cov_CRT_1}, and \eqref{eqn:cov_CRT_2}, we see that
\begin{equation}
	\begin{aligned}
		f(u,v) \equiv 0 \pmod{\ell} & \iff f_{(a,d)}(s,t) \equiv 0 \pmod{k} \\
		f(u,v) \leq x & \iff f_{(a,d)}(s,t) \leq \frac{(a,d)^2}{d^2} x. 
	\end{aligned}
\end{equation}
Note by \eqref{eqn:cov_CRT_2} that the congruence conditions are equivalent as $(d,k) =1$ and $\ell = dk$. Since $(t,k) = 1$ necessarily, we have therefore established that
\begin{equation}
|\cA_{\ell}(x, f; d)| = \Big|\cB_k\Big( \frac{(a,d)^2x }{d^2}, f_{(a,d)}\Big)\Big|. 
\label{eqn:local_density_correspondece}
\end{equation}
Summing this identity over $d \mid \ell$, we apply observation \eqref{eqn:sieve_sequence_observations} and \cref{lemma:primitive_congruence_sum} to deduce that
\begin{equation}
\begin{aligned}
	|\cA_{\ell}(x,f)| 
		& = \sum_{\ell = dk} \Big|\cB_{k}(\frac{(a,d)^2 x}{d^2}, f_{(a,d)}) \Big|	\\
		& =  \frac{\pi \sqrt{D}}{2a} V^2 \cdot\sum_{\ell = dk} M_{f_{(a,d)}}(k) \frac{\varphi(k)}{k^2} \frac{(a,d)}{d^2} + O\Big( V \cdot \sum_{\ell = dk} \frac{M_{f_{(a,d)}}(k)}{d}  \\
		& \qquad \qquad +  \frac{\sqrt{D}}{a}V^{1/2} \sum_{\ell = dk} M_{f_{(a,d)}}(k) k^{1/2} \tau(k) d^{-1/2} + \sum_{\ell = dk} M_{f_{(a,d)}}(k) \delta(k) \Big). 
\end{aligned}	
\label{eqn:congruence_sum_unsimplified}
\end{equation}
We wish to simplify the remaining sums and error term. Let $r \mid \ell$. As $f$ is primitive, 
\[
M_{f}(p) = \begin{cases} 
 	1 + \chi(p) & \text{if $p \nmid a$,} \\
 	\chi(p) & \text{if $p \mid a$,} \\
 \end{cases}
\]
by \eqref{def:M_solns}. To compute $M_{f_r}$, observe by the primitivity of $f$ that a prime $p$ divides $(a, br, cr^2)$ if and only if $p$ divides $(a,r)$. Moreover, if $p \mid r$ then $\chi_{r^2D}(p) = \big(\frac{r^2D}{p}\big) = \big( \frac{r^2}{p}\big) \big(\frac{D}{p}\big) = 0$ and, similarly, if $p \nmid r$ then $\chi_{r^2D}(p) = \chi(p)$.  Combining these observations with \eqref{def:M_solns} and \eqref{def:f_r}, we see that
\[
M_{f_{(a,d)}}(p) = \begin{cases} 
 	1 + \chi(p) & \text{if $p \nmid a$,} \\
 	\chi(p) & \text{if $p \mid a$ and $p \nmid (a,d)$,} \\
 	p & \text{if $p \mid (a,d)$}. 	
 \end{cases}
\]
In particular, as $(d,k) = 1$, it follows that $M_{f_{(a,d)}}(k) = M_f(k)$. Hence, 
\begin{equation}
\begin{aligned}
& \sum_{\ell = dk} M_{f_{(a,d)}}(k) \frac{\varphi(k)}{k^2} \frac{(a,d)}{d^2}\\
& \qquad = \sum_{\ell=dk} M_f(k) \cdot \frac{\varphi(k)}{k^2} \cdot \frac{(a,d)}{d^2} 	\\
& \qquad = \prod_{\substack{p \mid \ell \\ p \nmid a}} \Big( (1+\chi(p)) ( \frac{1}{p} - \frac{1}{p^2}) + \frac{1}{p^2} )\Big) \times \prod_{p \mid (\ell,a)} \Big( \chi(p)(\frac{1}{p}-\frac{1}{p^2}) + \frac{1}{p}\Big) \\
& \qquad = \prod_{p \mid \ell} \Big( \frac{1+\chi(p)}{p} - \frac{\chi(p)}{p^2}\Big) = g(\ell).
\end{aligned}	
\label{eqn:local_density_main_term}
\end{equation}
Similarly,
\begin{equation}
\begin{aligned}
\sum_{\ell = dk} \frac{M_{f_{(a,d)}}(k)}{d} 
& = \prod_{\substack{p \mid \ell \\ p \nmid a}} \Big(1 + \chi(p) + \frac{1}{p}\Big) \times \prod_{p \mid (\ell,a)} \Big( \chi(p) + \frac{1}{p} \Big)  \ll \tau_3(\ell),
\end{aligned}	
\label{eqn:local_density_error1}
\end{equation}
which implies that
\begin{equation}
\sum_{\ell = dk} M_{f_{(a,d)}}(k) k^{1/2} \tau(k) d^{-1/2} \ll \ell^{1/2} \tau(\ell)  \sum_{\ell = dk} \frac{M_{f_{(a,d)}}(k)}{d} \ll  \ell^{1/2} \tau(\ell) \tau_3(\ell). 
\label{eqn:local_density_error2}
\end{equation}
Combining the observation that
\[
\sum_{\ell = dk} M_{f_{(a,d)}}(k) \delta(k) = M_{f_{(a,d)}}(1) = 1
\]
with \eqref{eqn:congruence_sum_unsimplified}, \eqref{eqn:local_density_main_term}, \eqref{eqn:local_density_error1}, and \eqref{eqn:local_density_error2} yields the desired result. 
\end{proof} 
%

To obtain a better intuition for the quality of \cref{prop:local_density}, we present the special case when $\ell=1$ as a corollary below. We do not claim that this corollary is new, but we have not seen it stated in the literature and thought it may be of independent interest. 

%
%
%
\begin{corollary}
	\label{corollary:congruence_sum_trivial}	
	Let $f(u,v) = au^2 + buv + cv^2$ be a primitive positive definite binary quadratic form with discriminant $-D$. For $x \geq 1$, 
	\[
	\sum_{n \leq x} r_f(n) = \frac{2\pi x}{\sqrt{D}} + O\Big( \frac{ (ax)^{1/2}}{D^{1/2}}  +\frac{(Dx)^{1/4}}{a^{3/4}} + 1\Big).
	\]
\end{corollary}
%
\begin{remarks} Suppose $f$ is reduced so $|b| \leq a \leq c$. It is well known (see e.g. \cite[Lemma 3.1]{BlomerGranville-2006}) that
\[
\sum_{n \leq x} r_f(n) = \frac{2\pi x}{\sqrt{D}} + O\Big( \frac{x^{1/2}}{a^{1/2}} + 1\Big). 
\]
This estimate, like \cref{corollary:congruence_sum_trivial}, gives the asymptotic $\sim \frac{2\pi x}{\sqrt{D}}$ as long as $x/c \rightarrow \infty$, but the error term in \cref{corollary:congruence_sum_trivial} is stronger than the above whenever $x \geq c$. The source of this improvement is a standard analysis of the sawtooth function in \cref{lemma:sqrt_average} and its subsequent application in \cref{lemma:primitve_congruence_sum_fibre}. As discussed in \cref{subsec:optimality}, the condition $x \geq c$ is the `non-trivial' range for counting the lattice points inside the ellipse $f(u,v) \leq x$ whenever $f$ is reduced. 
\end{remarks}
%

%
%
\section{Application of Selberg's sieve}

We now apply Selberg's sieve to give an upper bound for the number of primes in a short interval represented by a reduced positive definite primitive integral binary quadratic form. We leave the calculation of the main term's implied constant unfinished as the final arguments vary  slightly for \cref{theorem:BT_uniform,theorem:BT_average}.

%
%
%
%
\begin{proposition}
	\label{prop:selberg_sieve_prelim}
	Let $f(u,v) = au^2 + buv + cv^2$ be a reduced positive definite integral binary quadratic form with discriminant $-D$. Let $(ax)^{1/2} \leq y \leq x$. Set
	\begin{equation}
	z = \Big(\frac{a}{Dx}\Big)^{1/4} y^{1/2} (\log y)^{-7}+1. 
	\label{def:sifting_variable}
	\end{equation}
	If $x \geq D/a$ then
	\begin{equation}
 	\pi_f(x) - \pi_f(x-y) < \Big\{ \frac{\log y}{\mathcal{J}} + O\big( (\log y)^{-1} \big) \Big\} \frac{\delta_f y}{h(-D) \log y},  
 	\label{eqn:selberg_sieve_prelim}
	\end{equation}
	where
	\[
	\mathcal{J} =  
		\begin{cases}
			\displaystyle \frac{1}{L(1,\chi)}\sum_{\ell < z} g(\ell) & \text{if $L(1,\chi) \geq (\log y)^{-2}$}, \\
			(\log y)^2 & \text{otherwise.}
		\end{cases}
	\]
	Here $\chi = \chi_{-D}$ is the corresponding Kronecker symbol and $g$ is the completely multiplicative function defined by \eqref{def:local_density}. 
\end{proposition}
%

%
\begin{proof} As $f$ is reduced, we have that $|b| \leq a \leq c$ and moreover $c \asymp D/a \geq \sqrt{D} \geq a$. We will frequently apply these properties while only mentioning that $f$ is reduced. 

Our argument is divided according to the size of $L(1,\chi)$. First, assume $L(1,\chi) < (\log y)^{-2}$. Let $w_{-D}$ be the number of roots of unity contained in $\mathbb{Q}(\sqrt{-D})$. Trivially, by \cref{corollary:congruence_sum_trivial}, 
\begin{equation*}
\begin{aligned}
\frac{w_{-D}}{\delta_f} ( \pi_f(x) - \pi_f(x-y) )
	 \leq \sum_{x-y < n \leq x} r_f(n) = \frac{2\pi y}{\sqrt{D}} + O\Big( \frac{(ax)^{1/2}}{D^{1/2}} \Big)
\end{aligned}
\end{equation*}
because $x \geq D/a$ and $f$ is reduced. Thus, by the class number formula
\begin{equation}
h(-D) = \frac{ w_{-D} \sqrt{D}}{2\pi} L(1,\chi) 
\label{eqn:class_number_formula}
\end{equation}
and our assumption on $L(1,\chi)$, 
\begin{equation*}
\begin{aligned}
\pi_f(x) - \pi_f(x-y) 
	& \leq L(1,\chi) \Big\{ 1 + O\Big( \frac{(ax)^{1/2}}{y}\Big) \Big\}  \frac{\delta_f y}{h(-D)} 
	 \ll \frac{y}{h(-D) (\log y)^2}. 
\end{aligned}
\end{equation*}
In the last step, we used that $y \geq (ax)^{1/2}$. This establishes \eqref{eqn:selberg_sieve_prelim} when $L(1,\chi) < (\log y)^{-2}$. Therefore, we may henceforth assume
\begin{equation}
L(1,\chi) \geq (\log y)^{-2}.
\label{eqn:assumption_L1chi}
\end{equation}
Defining $P = P(z) = \prod_{p \leq z} p$, it follows that
\begin{equation}
	\frac{w_{-D}}{\delta_f} (\pi_f(x) - \pi_f(x-y)) \leq  \sum_{\substack{x-y < n \leq x  \\ (n,P) = 1}} r_f(n) + \frac{w_{-D}}{\delta_f} \pi(z),
	\label{eqn:selberg_sieve_step0}
\end{equation}
where $\pi(z)$ is the number of primes up to $z$. We proceed to estimate the sieved sum. 
Using \cref{prop:local_density}	and Selberg's upper bound sieve \cite[Theorem 7.1]{FriedlanderIwaniec-2010} with level of distribution $z^2$, we see that
\begin{equation}
\sum_{\substack{x-y < n \leq x  \\ (n,P) = 1}} r_f(n) < \frac{2\pi y}{\sqrt{D} J}  + \sum_{\substack{\ell \mid P \\ \ell < z^2}} r_\ell \lambda_\ell ,
\label{eqn:selberg_sieve_step1}
\end{equation}
where 
\[
J = \sum_{\substack{\ell \mid P \\ \ell < z}} h(\ell), \qquad h(\ell	) = \prod_{p \mid \ell} \frac{g(p)}{1-g(p)}, \qquad |\lambda_{\ell}| \leq \tau_3(\ell), 
\]
and
\[
r_{\ell} \ll \tau_3(\ell) V + \ell^{1/2} \tau(\ell) \tau_3(\ell) \frac{\sqrt{D}}{a} V^{1/2}. 
\]
Here, as usual, $V = \sqrt{4ax/D}$. Note $1 \leq V \leq x^{1/2}$ as $x \geq D/a$ and $a \leq \sqrt{D}$. For the quantity $J$ in the main term, we treat $g$ as a completely multiplicative function and note that 
\begin{equation}
J \geq \sum_{\ell < z} g(\ell) = L(1,\chi) \mathcal{J} 
\label{eqn:selberg_sieve_J}
\end{equation}
by \eqref{eqn:assumption_L1chi}. The remainder term in \eqref{eqn:selberg_sieve_step1} is bounded in a straightforward manner. Using standard estimates for the $k$-divisor function $\tau_k(\ell)$ (see, e.g., \cite{LucaToth-2017}) and the prime number theorem, one can verify that
\begin{equation*}
\begin{aligned}
\frac{w_{-D}}{\delta_f} \pi(z) + \sum_{\substack{\ell \mid P \\ \ell < z^2}}r_{\ell} \lambda_{\ell}
		& \ll \frac{z}{\log z} + V \sum_{\ell < z^2} \tau_3(\ell)^2 +  \frac{\sqrt{D}}{a} V^{1/2} \sum_{\ell < z^2} \ell^{1/2} \tau(\ell) \tau_3(\ell)^2  \\
		& \ll   z^2 (\log z)^{8} \cdot V   + z^{3} (\log z)^{17} \cdot  \frac{\sqrt{D}}{a} V^{1/2} \\
		& \ll \frac{ay}{D} (\log y)^{-6} + \frac{y^{3/2} x^{-1/2}}{\sqrt{D}} (\log y)^{-4}.  
\end{aligned}	
\label{eqn:selberg_sieve_remainder}
\end{equation*}
In the last step, we used that $V = \sqrt{4ax/D}$ and, by \eqref{def:sifting_variable}, $z = (\frac{a}{Dx})^{1/4} y^{1/2} (\log y)^{-7}+1 \leq y$. Since $y \leq x$ and $a \leq \sqrt{D}$, we see that the above is 
\[
\ll \frac{y}{\sqrt{D} (\log y)^4}. 
\]
Thus, applying the class number formula \eqref{eqn:class_number_formula} and the well-known estimate $L(1,\chi) \ll \log D \ll \log y$, we conclude that
\begin{equation}
\begin{aligned}
\frac{w_{-D}}{\delta_f} \pi(z) + \sum_{\substack{\ell \mid P \\ \ell < z^2}}r_{\ell} \lambda_{\ell}
	 \ll  \frac{y}{h(-D) (\log y)^3}.	
\end{aligned}
\label{eqn:selberg_sieve_remainder_final}
\end{equation}
Combining \eqref{eqn:selberg_sieve_step0}, \eqref{eqn:selberg_sieve_step1}, \eqref{eqn:selberg_sieve_J}, and \eqref{eqn:selberg_sieve_remainder_final} completes the proof of the proposition with a final application of the class number formula.  
\end{proof}
%

Evidently, from \cref{prop:selberg_sieve_prelim}, we will require a lower bound for the sum of local densities. We execute the first steps here.

%
%
%
\begin{lemma}
	\label{lemma:sum_of_local_densities}
	Let $g$ be the completely multiplicative function defined by \eqref{def:local_density}. For $z \geq 1$, 
	\[
	\sum_{\ell  < z} g(\ell)  \geq  L(1,\chi)  \log z + L'(1,\chi) + O\Big(L(1,\chi) + \cE_1(z; \chi) + z^{-1}   \cE_0(z; \chi) \Big). 
	\]
	Here $\cE_1(z;\chi)$ and $\cE_0(z;\chi)$ are defined as in \cref{lemma:weighted_Dirichlet_sum}. 
\end{lemma}
\begin{proof}
Define
\begin{equation*}
\begin{aligned}
G(s) & := \sum_{n=1}^{\infty} g(n) n^{-s} = \prod_{p} \Big(1- g(p) p^{-s}\Big)^{-1}, \\
\end{aligned}
\end{equation*}
which absolutely converges for $\Re\{s\} > 0$ since $|g(p)| \leq 2/p$. 
One can verify that for $\Re\{s\} > 0$ 
\begin{equation}
G(s) = \zeta(s+1) L(s+1,\chi) \tilde{G}(s),
\label{eqn:local_density_decompose}
\end{equation}
where $\zeta(s)$ is the Riemann zeta function, $L(s,\chi)$ is the Dirichlet $L$-function attached to the quadratic character $\chi = \chi_{-D}$, and 
\[
\tilde{G}(s) := \prod_{p} \Big(1 - \frac{ \chi(p)  p^{-s-2} - \chi(p)  p^{-2s-2}}{1- (1+\chi(p)) p^{-s-1} + \chi(p) p^{-s-2}} \Big).  
\]
It is straightforward to check that $\tilde{G}(s)$ is absolutely convergent for $\Re\{s\} > -1$ and, in particular, $\tilde{G}(0) = 1$. Expanding the Euler product for $\tilde{G}(s)$ and writing $\tilde{G}(s) = \sum_n \tilde{g}(n) n^{-s}$ for some multiplicative function $\tilde{g}$, one can see that
\[
\tilde{g}(n) \ll  n^{-2}. 
\]
As $\tilde{G}(0) = 1$, it follows that
\[
\sum_{n \leq N} \tilde{g}(n) = 1 + O( N^{-1}).
\]
Therefore, from \eqref{eqn:local_density_decompose}, 
\begin{equation*}
	\begin{aligned}
		\sum_{\ell < z} g(\ell) 
			& = \sum_{\ell < z} \sum_{\ell = mn} \frac{(1 \ast \chi)(m)}{m} \tilde{g}(n)   = \sum_{m < z}  \frac{(1 \ast \chi)(m)}{m} + O\Big( z^{-1} \sum_{m < z} (1 \ast \chi)(m) \Big). \\
	\end{aligned}
\end{equation*}
The desired result now follows from \cref{lemma:weighted_Dirichlet_sum}. 
\end{proof}

%
%
\section{Representation of primes}

We may finally prove \cref{theorem:BT_average,theorem:BT_uniform}. In both cases, we will need to apply \cref{prop:selberg_sieve_prelim} from which  one can see that it suffices to provide an appropriate lower bound for $\mathcal{J}$ when $L(1,\chi) \geq (\log y)^{-2}$. By \cref{lemma:sum_of_local_densities}, it follows that
\begin{equation}
	\cJ \geq \log z + \frac{L'}{L}(1,\chi) + O\Big( 1 + \mathcal{E}_1(z; \chi) (\log y)^2 + \cE_0(z; \chi) \frac{(\log y)^2}{z} \Big)
	\label{eqn:J_lowerbound}
\end{equation}
provided $L(1,\chi) \geq (\log y)^{-2}$, $(ax)^{1/2} \leq y \leq x$,  and $z$ is given by \eqref{def:sifting_variable}. Recall that $\cE_0$ and $\cE_1$ are given by \cref{lemma:weighted_Dirichlet_sum} with estimates exhibited in \cref{sec:average_bounds,sec:uniform_bounds} . The proofs for both theorems will employ \eqref{eqn:J_lowerbound}. 

Before we proceed, we wish to emphasize that $f(u,v) = au^2 + buv + cv^2$ is assumed to be a \emph{reduced} positive definite binary integral quadratic form of discriminant $-D$. Thus, $|b| \leq a \leq c$ and $a \leq \sqrt{D}$. 

%
%
%
\subsection{Proof of \cref{theorem:BT_uniform}}
\label{subsec:proof_BT_uniform}
Recall we are assuming that 
\[
\big(\frac{D^{1+4\phi}}{a}\big)^{1/2+\epsilon} x^{1/2+\epsilon} \leq y \leq x,
\]
where $\phi$ is given by \eqref{def:phi}. As $a \leq \sqrt{D}$, this implies that $y \geq (ax)^{1/2}$. Furthermore, 
\[
z =  \Big(\frac{a}{Dx}\Big)^{1/4} y^{1/2} (\log y)^{-7}+1 \gg_{\epsilon} D^{\phi+\epsilon/2}, \qquad 
		\text{and} 
	\qquad
		\log z \asymp \log y \asymp \log x. 
\]
Therefore, applying \cref{lemma:log_derivative_unconditional} (or \cref{lemma:log_derivative_GRH} when assuming GRH) and \cref{lemma:uniform_bounds} to \eqref{eqn:J_lowerbound}, it follows that
\begin{equation*}
\begin{aligned}
\mathcal{J} 
	& \geq \log z -\Big(\frac{\phi}{2} + \frac{\epsilon}{4}\Big) \log D + O_{\epsilon}(1 + z^{-\epsilon^2} \log^2 y) \\
	& \geq \frac{1}{2} \log y - \frac{1}{4} \log x - (\frac{1}{4} + \frac{\phi}{2} + \frac{\epsilon}{4}) \log D + \frac{1}{4} \log a + O_{\epsilon}(\log\log y) \\
	& = \frac{1 - \theta}{2}\log y + O_{\epsilon}(\log\log y),
\end{aligned}	
\end{equation*}
where $\theta$ is defined as in \cref{theorem:BT_uniform}. Substituting this estimate in \cref{prop:selberg_sieve_prelim} establishes \cref{theorem:BT_uniform} when $L(1,\chi) \geq (\log y)^{-2}$. If $L(1,\chi) < (\log y)^{-2}$ then the desired result follows immediately from \cref{prop:selberg_sieve_prelim} and hence completes the proof. \hfill \qed

%
%
%
\subsection{Proof of \cref{theorem:BT_average}}
\label{subsec:proof_BT_average}

Recall $\mathfrak{D}(Q)$ is given by \eqref{def:discriminants}. Let $c_1(\epsilon) > 0$ be a sufficiently small constant and $C_1(\epsilon), C_2(\epsilon) \geq 1$ be sufficiently large constants, all of which depend only on $\epsilon$. For $Q \geq 3 $, let $\mathfrak{D}_{\epsilon}(Q)$ be the subset of discriminants $-D \in \mathfrak{D}(Q)$ such that
\begin{equation}
	 -\frac{L'}{L}(1,\chi_{-D}) \leq \epsilon \log D + C_2(\epsilon)
	 \label{eqn:log_derivative_average}
\end{equation}
and, for $c_1(\epsilon) D^{\epsilon} \leq u \leq C_1(\epsilon) D^{2+\epsilon}$,
\begin{equation}
	\cE_0(u; \chi_{-D}) \leq u^{7/8+\epsilon}, \qquad 
	\cE_1(u; \chi_{-D}) \leq u^{-1/8+\epsilon}. 
	\label{eqn:errors_average}
\end{equation}
By \cref{lemma:charsum_errors_average,lemma:log_derivative_average}, the number of discriminants \emph{not} satisfying these two properties is  
\[
|\mathfrak{D}(Q) \setminus \mathfrak{D}_{\epsilon}(Q)| \ll_{\epsilon} Q^{1-\epsilon/10}.
\]
Thus, it suffices to show for every discriminant $-D \in \mathfrak{D}_{\epsilon}(Q)$ and reduced positive definite binary quadratic form $f$ of discriminant $-D$ that
\begin{equation}
\pi_f(x) - \pi_f(x-y)  < \frac{2}{1-\theta'} \frac{\delta_f y}{\log y} \Big\{ 1 + O_{\epsilon}\Big(\frac{\log\log y}{\log y}\Big) \Big\}
\label{eqn:GRH_average_bound}
\end{equation}
provided  $\big(\frac{Dx}{a})^{1/2+\epsilon} \leq y \leq x$. Here $\theta'$ is defined as in \cref{theorem:BT_uniform} with $\phi = 0$. First, assume 
\begin{equation}
\big(\frac{D^2 x}{a}\big)^{1/2+\epsilon}  \leq y \leq x. 
\label{eqn:GRH_average_bound_range1}
\end{equation}
Arguing as in \cref{subsec:proof_BT_uniform}, it follows that $y \geq (ax)^{1/2}$,
\[
z =  \Big(\frac{a}{Dx}\Big)^{1/4} y^{1/2} (\log y)^{-7}+1 \gg_{\epsilon} D^{1/4+\epsilon/2}, \qquad 
		\text{and} 
	\qquad
		\log z \asymp \log y \asymp \log x. 
\]
Thus, incorporating \eqref{eqn:log_derivative_average} and \eqref{eqn:uniform_bounds_unconditional} from  \cref{lemma:uniform_bounds} into \eqref{eqn:J_lowerbound}, it similarly follows that
\begin{equation}
\mathcal{J} \geq \frac{1-\theta'}{2} \log y + O_{\epsilon}(\log\log y)
\label{eqn:J_lowerbound_average}
\end{equation}
whenever $L(1,\chi) \geq (\log y)^{-2}$. Therefore, by \cref{prop:selberg_sieve_prelim}, this establishes \eqref{eqn:GRH_average_bound} provided \eqref{eqn:GRH_average_bound_range1} holds and $-D \in \mathfrak{D}_{\epsilon}(Q)$. It remains to consider the case when
\begin{equation}
\big(\frac{D x}{a}\big)^{1/2+\epsilon}  \leq y \leq \big(\frac{D^2 x}{a}\big)^{1/2+\epsilon} \leq x. 
\label{eqn:GRH_average_bound_range2}
\end{equation}
Note we continue to assume $-D \in \kD_{\epsilon}(Q)$. As before, it follows that $y \geq (ax)^{1/2}$,
\[
z =  \Big(\frac{a}{Dx}\Big)^{1/4} y^{1/2} (\log y)^{-7}+1 \gg_{\epsilon} D^{\epsilon/2}, \qquad 
		\text{and} 
	\qquad
		\log z \asymp \log y \asymp \log x. 
\]
Thus, incorporating \eqref{eqn:log_derivative_average} and \eqref{eqn:errors_average} into \eqref{eqn:J_lowerbound}, we again obtain \eqref{eqn:J_lowerbound_average} whenever $L(1,\chi) \geq (\log y)^{-2}$. By \cref{prop:selberg_sieve_prelim}, this establishes \eqref{eqn:GRH_average_bound} provided \eqref{eqn:GRH_average_bound_range2} holds and $-D \in \mathfrak{D}_{\epsilon}(Q)$. This completes the proof in all cases. \hfill \qed 

%
%
\section{Representation of small integers with few prime factors}

\noindent
\emph{Proof of \cref{theorem:almost_primes}.}
We apply the beta sieve to the sequence $\mathcal{A} = \mathcal{A}(x,f)$ given by \eqref{def:sieve_sequence}. Let $g$ be the local density function defined in \cref{prop:local_density}, so $g(p) \leq 2/p$. Thus, the sequence $\mathcal{A}$ is of dimension at most $\kappa = 2$ and has sifting limit $\beta = \beta(\kappa) < 4.85$ according to \cite[Section 11.19]{FriedlanderIwaniec-2010}. Let $x \geq D/a$ and select
	\[
	z := V^{10/49}.
	\]
	where $V = \sqrt{4ax/D} \geq 2$. Select the level of distribution to be $R = z^{485/100} > z^{\beta}$.  Thus, by \cite[Theorem 11.13]{FriedlanderIwaniec-2010} and \cref{prop:local_density}, it follows that
	\begin{equation}
		\sum_{\substack{n \leq x \\ (n,P(z)) = 1} } r_f(n)
			 \gg \frac{x}{\sqrt{D} (\log x)^2}  + O\Big( \sum_{ \substack{\ell \mid P(z) \\ \ell < R} } |r_{\ell}| \Big),
	\end{equation}
	where 
	\[
	|r_{\ell}| \ll \ell^{\eta} V + \ell^{1/2+\eta} \frac{\sqrt{D}}{a} V^{1/2} + 1
	\]
	for fixed $\eta > 0$ sufficiently small. Since $f$ is reduced and $R = z^{485/100} = V^{97/98}$, we see that
	\begin{equation*}
		\sum_{ \substack{\ell \mid P(z) \\ \ell < R} } |r_{\ell}| 
			 \ll  R^{1+\eta} V + \frac{\sqrt{D}}{a} R^{3/2+\eta} V^{1/2} + R^{1-\eta} 
			 \ll  \frac{\sqrt{D}}{a} V^{2-\frac{3}{196}+\eta} \ll \frac{x^{1-\frac{3}{392} +\eta}}{\sqrt{D}}.
	\end{equation*}
	Thus, as $\eta > 0$ is sufficiently small,
	\begin{equation*}
		\sum_{\substack{n \leq x \\ (n,P(z)) = 1} } r_f(n)
			 \gg \frac{x}{\sqrt{D} (\log x)^2}
	\end{equation*}
	for $x \geq D/a$. For an integer $k \geq 10$, observe that $z \geq x^{1/k}$ if and only if
	\[
	\Big( \frac{ax}{D}\Big)^{5k/49} \geq x
	 \iff x \geq \Big(\frac{D}{a}\Big)^{1+\frac{49}{5k-49}}.
	\]
	This completes the proof. \hfill \qed

\bibliographystyle{alpha}
\bibliography{bibtex_library}

\end{document}